\documentclass[a4paper,12pt]{article}

\usepackage[margin=3cm]{geometry}
\usepackage[pdftex]{graphicx}
\usepackage[utf8]{inputenc}
\usepackage{amsthm}
\usepackage{amsmath}
\usepackage{amssymb}
\usepackage[dvipsnames]{xcolor}
\usepackage{authblk}
\usepackage{enumerate}

\usepackage[parfill]{parskip}

\usepackage[color links=true, backref=page]{hyperref}

\usepackage{soul}

\usepackage[noend]{algpseudocode}
\usepackage{arydshln}
\usepackage{pythonhighlight}

\usepackage{subcaption}

\usepackage{tikz}
\usetikzlibrary{calc}
\usetikzlibrary{backgrounds}
\usetikzlibrary{shapes}
\usetikzlibrary{positioning}
\usetikzlibrary{decorations, decorations.pathmorphing, decorations.text}
\usetikzlibrary{perspective}

\theoremstyle{plain}
\newtheorem{theorem}{Theorem}

\newtheorem{lemma}[theorem]{Lemma}

\theoremstyle{definition}
\newtheorem{definition}{Definition}

\newtheorem{remark}{Remark}

\newtheorem{algorithm}[theorem]{Algorithm}

\newcommand{\mc}[1]{\mathcal{#1}}
\newcommand{\ZZ}{\mathbb{Z}}

\newcommand{\bpm}{\begin{pmatrix}}
\newcommand{\epm}{\end{pmatrix}}

\title{Polycyclic Geometric Realizations of the Gray Configuration}
\author[1]{Leah Wrenn Berman}
\author[2]{G\'{a}bor G\'{e}vay}
\author[3,4,5]{Toma\v{z} Pisanski}

\affil[1]{Department of Mathematics \& Statistics,
University of Alaska Fairbanks,
Fairbanks, Alaska, USA}
\affil[2]{Bolyai Institute, University of Szeged, Szeged, Hungary}
\affil[3]{FAMNIT, University of Primorska, Koper, Slovenia}
\affil[4]{IAM, University of Primorska, Koper, Slovenia}
\affil[5]{Institute of Mathematics, Physics and Mechanics, Ljubljana, Slovenia}

\date{\empty}

\begin{document}

\maketitle

\begin{abstract}

The Gray configuration is a $(27_3)$ configuration which typically is realized as the points and lines of the $3 \times 3 \times 3$ integer lattice. It occurs as a member of an
infinite family of configurations defined by Bouwer in 1972. Since their discovery, both the Gray 
configuration and its Levi graph (i.e., its point-line incidence graph) have been the subject of 
intensive study. Its automorphism group contains cyclic subgroups isomorphic to  $\mathbb Z_3$ and $\mathbb 
Z_9$, so it is natural to ask whether the Gray configuration can be realized in the plane with any of the 
corresponding rotational symmetry. In this paper, we show that there are two distinct polycyclic realizations with $\mathbb{Z}_{3}$ symmetry. In contrast, the only geometric polycyclic realization with straight lines and $\mathbb{Z}_{9}$ symmetry is only a ``weak'' realization, with extra unwanted incidences (in particular, the realization is actually a $(27_{4})$ configuration).

\end{abstract}

\smallskip

\noindent
{\bf Math.\ Subj.\ Class.\ (2020):} 05B30, 05C10, 05C60, 51A45, 51E30.

\section{Introduction}


The \emph{Gray graph} was discovered by Marion C.\ Gray in 1932, and was rediscovered 
independently by Bouwer when searching for regular graphs that are edge-transitive but not 
vertex-transitive~\cite{Bou1972}; graphs fulfilling these conditions are called 
\emph{semisymmetric}~\cite{MoPi2007}), so henceforth we use this term. 
A first detailed study of semisymmetric graphs is due to Folkman~\cite{Fol1967}. The Gray graph, in 
particular, also became the subject of a careful investigation: the third author of this paper and his 
co-workers explored many interesting properties of this graph~\cite{MaPi2000, MaPiWi2005, 
MoPi2007, Pis2007}; see also~\cite[Chapter 6]{PisSer2013} (for more details about its history, 
see~\cite{MoPi2007}). We know that it is the smallest trivalent semisymmetric graph. Its girth is 
8, which is equivalent to saying that the Gray configuration is triangle-free. It is Hamiltonian, with a
Hamiltonian realization displaying a $\mathbb Z_9$ symmetry, as a construction due to Randi\'c reveals 
it~\cite{MaPi2000,PiRa2000}. A corresponding LCF notation is $[7, -7, 13, -13, 25, -25]^9$. 
This can be clearly seen in Figure~\ref{fig:GrayGraph}.
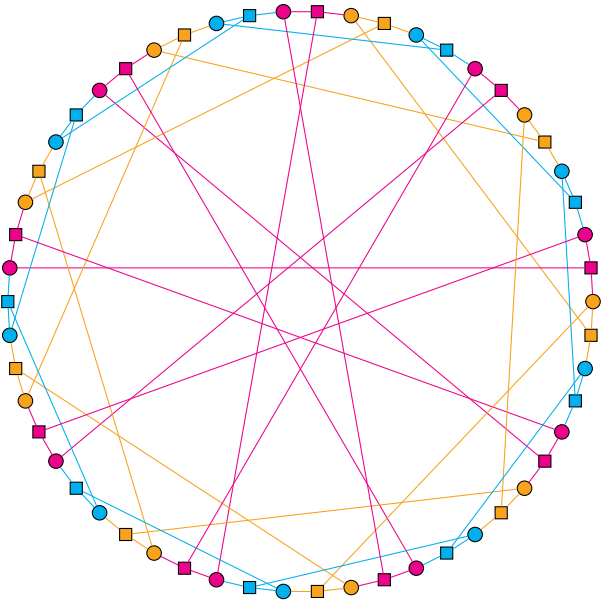
\begin{figure}[!h]
\begin{center}
\begin{tikzpicture}[vtx/.style={draw, circle, inner sep = 2.75 pt}, line/.style = {draw, inner sep = 3.25 pt}, scale = .7, transform shape]
\pgfmathsetmacro{\r}{5.5}
\newcommand{\colR}{Magenta}
\def\colG{ProcessBlue}
\newcommand{\colB}{YellowOrange}
\foreach \i in {0,...,8}{\node[vtx, fill = \colB, 
] (B\i) at (360*\i/9:\r){};}
\foreach \i in {0,...,8}{\node[vtx, fill = \colG] (G\i) at (360*\i/9- 40/3:\r){};}
\foreach \i in {0,...,8}{\node[vtx, fill = \colR] (R\i) at (360*\i/9- 2*40/3:\r){};}
\foreach \i in {0,...,8}{\node[line, fill = \colB
] (b\i) at (360*\i/9-360/54:\r){};}
\foreach \i in {0,...,8}{\node[line, fill = \colG] (g\i) at (360*\i/9- 40/3 - 360/54:\r){};}
\foreach \i in {0,...,8}{\node[line, fill = \colR] (r\i) at (360*\i/9- 2*40/3-360/54:\r){};}
\foreach \i in {0,..., 8}{
\draw[\colB] let \n1 = {int(mod(\i+2,9))} in (b\i) -- (B\n1);
\draw[\colB] let \n1 = {int(mod(\i+0,9))} in (b\i) -- (B\n1);
\draw[\colB] let \n1 = {int(mod(\i+0,9))} in (b\i) -- (G\n1);
\draw[\colG] let \n1 = {int(mod(\i+1,9))} in (g\i) -- (G\n1);
\draw[\colG] let \n1 = {int(mod(\i+0,9))} in (g\i) -- (G\n1);
\draw[\colG] let \n1 = {int(mod(\i+0,9))} in (g\i) -- (R\n1);
\draw[\colR] let \n1 = {int(mod(\i+4,9))} in (r\i) -- (R\n1);
\draw[\colR] let \n1 = {int(mod(\i+0,9))} in (r\i) -- (R\n1);
\draw[\colR] let \n1 = {int(mod(\i+8,9))} in (r\i) -- (B\n1);
}
\end{tikzpicture}
\caption{The Gray graph. The square nodes correspond to lines and the circular nodes correspond to points; the 
coloring demonstrates  \emph{rotational} $\mathbb{Z}_{9}$ symmetry. This is further explored in Section~\ref{sect:9fold}.}
\label{fig:GrayGraph}
\end{center}
\end{figure}

The \emph{Gray configuration} is a $(27_{3})$ configuration which occurs as a member of an infinite 
family of configurations defined by Bouwer in 1972~\cite[Section 1]{Bou1972}. Its name stems from 
the fact that its Levi graph---that is, the point-line incidence graph---is the Gray graph; see Figure~\ref{fig:GrayGraph}. 
(While it may appear that the graph has dihedral symmetry, this is not the case, because the mirror symmetry interchanges 
line-nodes and point-nodes, and while the individual color classes have mirror symmetry, the graph as a whole does not.)  
By the definition due to Bouwer, it can be realized as a spatial 
configuration consisting of the 27 points and 27 lines of the $3\times3\times3$ integer grid (cf.\ 
Figure~\ref{fig:GraySpatial}). It can also be conceived as the Cartesian product of three copies 
of the ``dual pencil'' $(3_1,1_3)$ configuration, or equivalently, the Cartesian product of the dual 
pencil $(3_1,1_3)$ and the $(9_2,6_3)$ ``grid configuration''~\cite{Gev2013}. Together with its 
dual, it forms a pair of the smallest configurations which are triangle-free and flag transitive but 
not self-dual~\cite{MaPiWi2005}. Moreover, it is \emph{resolvable}. By definition, this means that 
the set of configuration lines partitions into classes (called \emph{resolution classes} or \emph{parallel 
classes}) such that within each class, the lines partition the set of points of the configuration by 
incidences~\cite{Gev2019}. This is clearly seen in Figure~\ref{fig:GraySpatial}, since the parallel 
classes of the lines coincide with the parallel classes in geometric sense. The grid structure makes 
possible assigning labels to the configuration points of the form $xyz$ $(x,y,z\in\mathbb Z_3)$ such 
that two points with labels $xyz$ and $x'y'z'$ are incident to the same line if and only if precisely two 
of the equalities $x=x'$, $y=y'$, $z=z'$ hold.
\begin{figure}[!h]
\begin{center}
\includegraphics[width=0.33\textwidth]{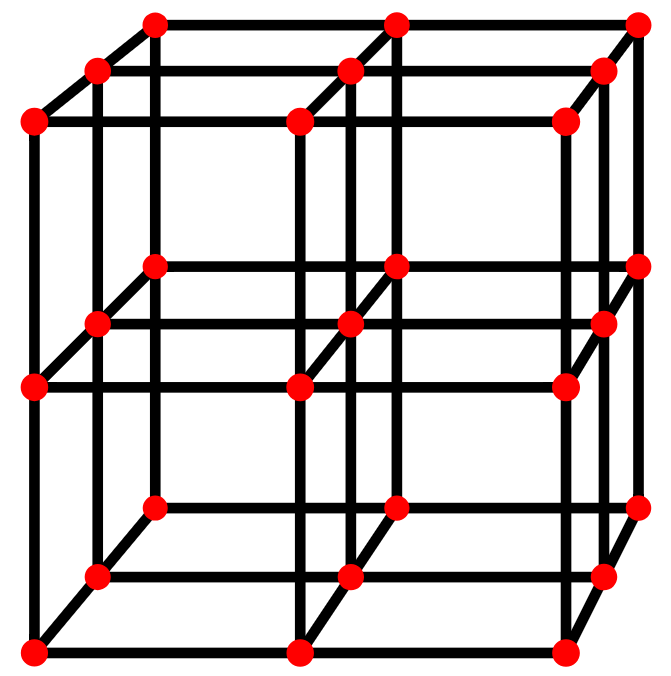}
\caption{Spatial realization of the Gray configuration. 
}
\label{fig:GraySpatial}
\end{center}
\end{figure}

A \emph{polycyclic geometric realization} of a configuration of points and lines is one in which (1) the 
combinatorial lines of the configuration are represented in the Euclidean plane using straight lines; 
(2) the points and lines are divided into symmetry classes in which each symmetry class contains 
the same number of elements. That is, it is a realization of the configuration in which a semi-regular 
subgroup of the automorphism group maps the realization to itself.

The main results of this paper are
\begin{itemize}
\item 
to show that the Gray configuration can be realized polycyclically in two different ways with 
$\mathbb{Z}_{3}$ symmetry, realized as geometric rotation;
\item 
to show that the Gray configuration can only be weakly polycyclically geometrically realized with $
\mathbb{Z}_{9}$ symmetry (any straight-line realization forces extra incidences), but it can be topologically realized as a polycyclic realization using pseudolines;
\item 
There are no other geometric polycyclic realizations of the Gray configuration.
\end{itemize}


\section{The automorphism group of the Gray configuration}


The automorphism group $\mathrm{Aut}\,G$ of the Gray graph is a group of order $1296 = 6^{4}$, 
which can be given in the following form:
\begin{equation} \label{eq:autG}
(S_3 \times S_3 \times S_3) \rtimes S_3 \cong S_3 \wr S_3,
\end{equation}
where $S_3$ is the symmetric group of degree 3. This has been established in the literature in a 
more general setting, see e.g.~\cite{Pis2007}. The automorphism group of the Gray configuration 
is the same, since the Gray configuration is not self-dual (so there are no color-exchanging 
automorphisms). Here we give an independent proof, focusing directly on simple geometric 
properties of the spatial realization of the configuration.

In this spatial realization, the configuration contains three pencils of parallel layers. Each of these layers 
forms a $(9_2, 6_3)$ ``grid" configuration. Within each parallel pencil there are three layers, and these 
pencils are perpendicular (say) to the $x$-, $y$-, respectively the $z$-axis of a Cartesian coordinate 
system (note that the labels introduced in the previous section can be conceived as coordinates 
with respect to such a coordinate system in the Euclidean 3-space). Each of the three copies of the 
$S_3$ group in the parenthesis of (\ref{eq:autG}) is responsible for permuting the layers within 
a given pencil (independently of each other, which explains why direct products are used). 
Note that this product in the parenthesis does not move one pencil to any other pencil. 
On the other hand, the second term of the semidirect product is responsible for permuting 
the three pencils, each as a whole (while leaving fixed the order of the layers within a pencil). 
In addition, it is easily seen that the group on the left (i.e.\ the triple direct product in the parenthesis) 
is an invariant subgroup of the entire group (while this does not hold for the right term). 
This explains why we use here semidirect product. Finally, it is well known that the semidirect 
product of the form above can be rewritten in the form of a wreath product of two copies of $S_3$.


\section{Finding possible quotient graphs and reduced Levi graphs from semi-regular subgroups}


Although we are interested mainly in  polycyclic configurations, we first briefly consider  the general case of  semi-regular groups. 

The automorphism group $\mathrm{Aut}\, G$ of a simple graph $G$ may be viewed as a group of permutations acting 
on the set of vertices of $G$. Its orbits induce a partition of the set of vertices of $G$. The action of a nontrivial 
subgroup $\Gamma \subseteq \mathrm{Aut}\, G$ is \emph{semi-regular} if all of its orbits have the same cardinality 
$|\Gamma|$. Note that this is equivalent to saying that the stabiliser of $\Gamma$ is trivial; see  \cite{DoMaMa}. 
For more information on this and what follows, see \cite{DoMaMa,PisSer2013}.

 A subgroup of $\text{Aut}\,G$ with a semi-regular action on the vertices of $G$ is called a \emph{semi-regular subgroup}. 
All nontrivial subgroups of a semi-regular group are semi-regular. In particular,  each nontrivial element $\alpha$ 
of a semi-regular group is a generator of a semi-regular cyclic group. Hence $\alpha$ is a \emph{semi-regular} permutation.
This is equivalent to saying that when writing a permutation $\alpha$ as a product of disjoint cycles, the length of all cycles is the same. 

For any simple graph $G$ the action of $\Gamma \subseteq \mathrm{Aut}\, G$ on the vertex set extends to the actions on the 
set of edges (that is, undirected pairs of adjacent vertices) and the set of arcs or darts (that is, directed pairs of adjacent vertices).
The action of a semi-regular subgroup can be described via quotient graphs.  If $\Gamma$ is a semi-regular subgroup, then the graph 
$G/\Gamma$ is the quotient graph, with vertices in $G/\Gamma$ corresponding to orbits of vertices of $G$ under the action of 
$\Gamma$, and edges (including parallel edges, semi-edges, loops) determined by orbits of arcs  of $G$, in the standard way, and 
the projection $\pi: G \rightarrow G/\Gamma$ is called a \emph{regular covering projection} (see \cite{GrTu,PisSer2013,DoMaMa}). 
In the special case of a cyclic group $\Gamma$, we may consider its generator permutation $\alpha$.  If $\alpha$ is a semi-regular 
automorphism, then the graph $G/\alpha$ is the quotient graph, with vertices in $G/\alpha$ corresponding to orbits of vertices of 
$G$ under the action of $\alpha$, etc.

\begin{definition}
A simple graph $G$ of order $n$ admitting a semi-regular automorphism $\alpha$ of order $m$ is \emph{polycirculant}. If $m = n$,
the automorphism $\alpha$ is \emph{regular}.
\end{definition} 

\begin{remark}
Note that in the Definition above, $m$ divides $n$, and the quotient graph $G/\alpha$ has order $k$, where $n = mk$.
\end{remark}

The renowned theorem of 
Sabidussi~\cite{Sab1958} states that a graph $G$ is a Cayley graph if and only if $\mathrm{Aut}\, G$ admits a 
subgroup with regular action, which corresponds to the quotient graph $G/\Gamma$ having only a single vertex.

\begin{remark}
Note that a semi-regular automorphism of a simple graph acts semi-regularly on vertices and arcs of a graph but not 
necessarily on the edges of a graph. The easiest way to check whether the action on the edges is semi-regular is via 
quotient graphs; the action is semi-regular on the edges if and only if the quotient graph has no semi-edges.
\end{remark}

Given a graph $G$, it is possible to determine all distinct regular covering projections  
$\pi: G \rightarrow G/\Gamma$. The recipe is as follows.

First we generate all semi-regular group actions up to conjugacy. This can be easily done by 
SageMath/python/GAP, which has all ingredients readily available.
\bigskip

\begin{python}
def generate_semi_regular_actions(G):
    """Given graph G, generate all semi-regular group actions."""
    AutG = G.automorphism_group()
    for Gamma in AutG.conjugacy_classes_subgroups(): 
    	if Gamma.is_semi_regular():
    		yield Gamma
\end{python}

Using Sage \cite{sagemath} and the above algorithm, we determined that the Gray graph has 5 conjugacy classes of subgroups which produce bipartite quotient 
graphs; the groups are listed in Table \ref{tab:subgroups} and the corresponding quotient graphs in Figure \ref{fig:quotients}.
Note that the Gray graph is unusual in the sense that all quotients arising from semi-regular actions are bipartite!

\begin{table}[!h]
\caption{Semi-regular subgroups of $\mathrm{Aut}\, G $ for the Gray graph.}
\begin{center}
\begin{tabular}{r r l r}
ID & order & subgroup isomorphic to... & bipartite?\\ 
1& 3 & $\mathbb{Z}_{3}$& YES\\
2 & 3 & $\mathbb{Z}_{3}$& YES\\
3 & 9& $\mathbb{Z}_{3} \times \mathbb{Z}_{3}$& YES\\
4 & 9 & $\mathbb{Z}_{9}$& YES\\
5 & 27 & $\mathbb{Z}_{9}\times \mathbb{Z}_{3}$& YES\\
\end{tabular}
\end{center}
\label{tab:subgroups}
\end{table}

\bigskip\bigskip

\begin{figure}[!h]
\begin{center}
\begin{subfigure}[t]{.3\linewidth}
\begin{center}
\includegraphics[width=4cm]{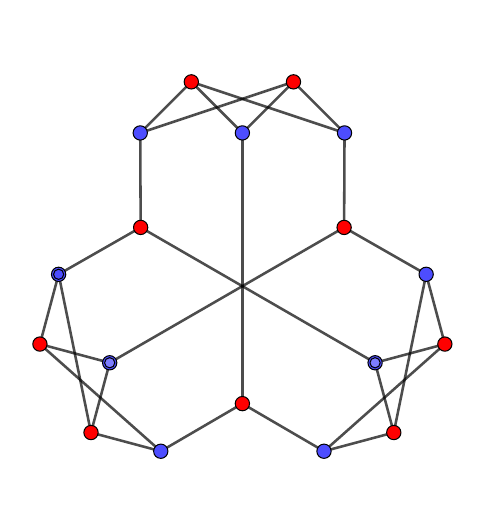}
\end{center}
\caption{Quotient with respect to $\ZZ_{3}$; the graph $GG$ (ID\#1)\label{fig:GG}}
\label{fig:grayZ3-GG}
\end{subfigure}
\hfill 
\begin{subfigure}[t]{.3\linewidth}
\centering
\includegraphics[width=5cm]{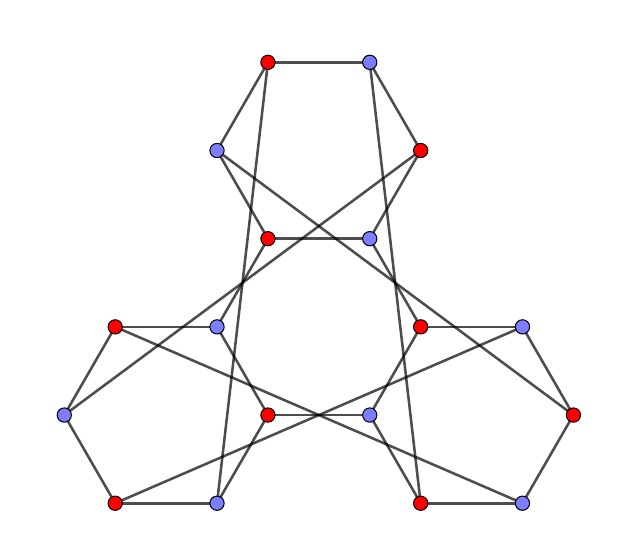}
\caption{Quotient with respect to $\ZZ_{3}$; the Pappus graph (ID\#2)\label{fig:Pappus}}
\label{fig:grayZ3-Pappus}
\end{subfigure}
\hfill
\begin{subfigure}[t]{.3\linewidth}
\centering
\includegraphics[width=2.8cm]{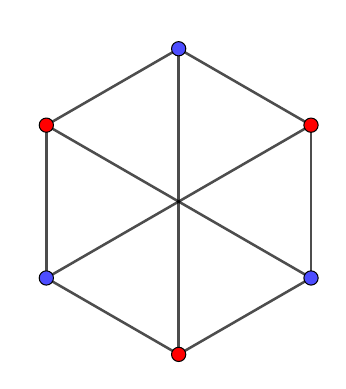}
\caption{Quotient with respect to $\ZZ_{3}\times \ZZ_{3}$ (ID\#3)\label{fig:K33}}
\label{fig:grayZ3xZ3}
\end{subfigure}

\begin{subfigure}[t]{.3\linewidth}
\includegraphics[width=3cm]{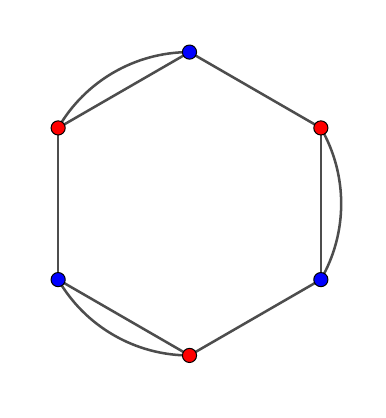}
\caption{Quotient with respect to $\ZZ_{9}$ (ID\#4)}
\label{fig:grayZ9}
\end{subfigure}
\begin{subfigure}[t]{.3\linewidth}
\includegraphics[width=3cm]{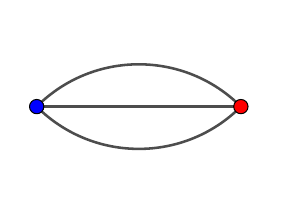}
\caption{Quotient with respect to $\ZZ_9 \rtimes \ZZ_3$. (ID\#5)}
\label{fig:grayZ9xZ3}
\end{subfigure}

\caption{The five semi-regular quotient graphs of the Gray graph,  and their IDs from Table
\ref{tab:subgroups}.
Since all quotients are bipartite, these graphs are all the possible RLGs of the Gray configuration.}
\label{fig:quotients}
\end{center}
\end{figure}

Given a configuration, the incidence graph of the configuration, usually called a \emph{Levi graph}, is formed by assigning one vertex of the graph to each point and line of the configuration, and joining vertices with edges if and only if the corresponding point and line are incident. 

\begin{definition}A geometric realization of a configuration is \emph{polycyclic} if there exists a semi-regular color-preserving automorphism $\alpha$ of the Levi graph $G$ that is realized by geometric rotation of the same order.
\end{definition}

A \emph{reduced Levi graph} (RLG) is a bipartite quotient of the Levi graph with a semi-regular subgroup of the automorphism group. We use the convention that arcs in reduced Levi graphs are directed from line-orbits to point-orbits. If you reverse all arrows in a reduced Levi graph and switch the interpretation of colors of nodes, the result is the reduced Levi graph of the dual configuration.

In the remainder of the paper, 
we determine which of the quotients of the Gray graph can be the reduced Levi graphs of polycyclic realizations of the Gray configuration, and we determine whether such geometric realizations exist.


\section{Labeling elements of the Gray configuration and identifying the reduced Levi graphs}


In what follows, we use the following labeling conventions. We label the points of the $3\times 3$ 
integer grid using the labels $ijk$, $i,j,k = 0,1,2$ corresponding to the axes in $\mathbb{R}^{3}$, 
with the first, second, third coordinates corresponding to left-right, down-up, and front-back 
respectively.

\newcommand{\myast}{\ensuremath{*}}

The lines of the configuration are labeled 
$* ij = \{0ij, 1ij, 2ij\}$, 
$i \!*\! j = \{i0j, i1j, i2j\}$, and 
$ij* = \{ij0, ij1, ij2\}$. See Figure~\ref{fig:GrayGrid}.

\def\lw{.6}
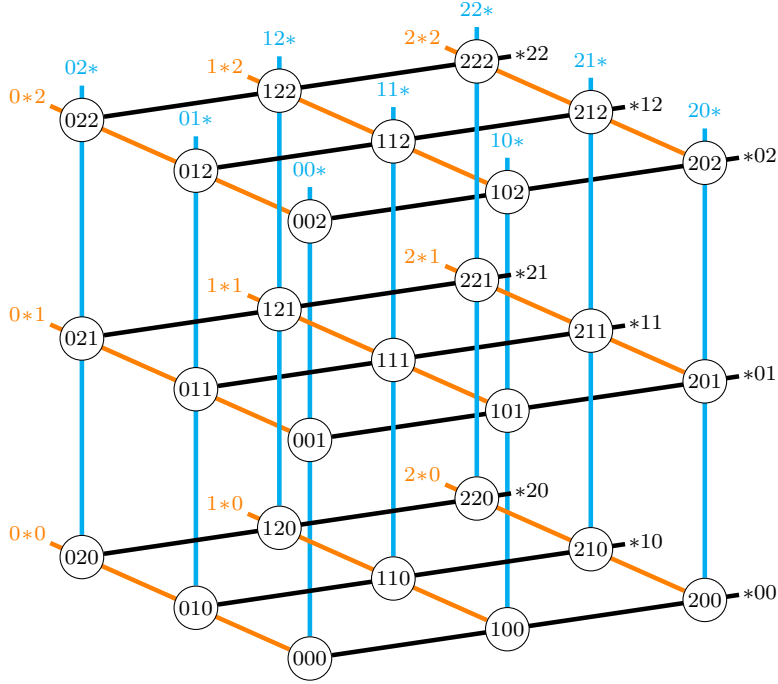
\begin{figure}[!h]
\begin{center}

\begin{tikzpicture}  [scale = 1,
  3d view,perspective
,
 ]  
\foreach \i in {0,1,2}{
\foreach \j in {0,1,2}{
\foreach \k in {0,1,2}{
\node[draw, circle, font = \scriptsize, fill = white, inner sep = 1pt] (\i\j\k) at ($3*(\i,\j,\k)$) {$\i\j\k$};
}}}

\begin{scope}[on background layer]
\foreach \i in {0,1,2}{
\foreach \j in {0,1,2}{
\draw [cyan,shorten >=-.75cm, line width = \lw mm] (\i\j0) -- (\i\j2) node[font = \scriptsize, pos = 1.19] {$\i\j*$};
\draw [orange,shorten >=-.75cm, line width = \lw mm] (\i0\j) -- (\i2\j) node[font = \scriptsize, pos = 1.4] {$\i\!*\!\j$};
\draw [,shorten >=-.75cm, line width = \lw mm] (0\i\j) -- (2\i\j) node[font = \scriptsize, pos = 1.22] {$*\i\j$};
}}
\end{scope}
\end{tikzpicture}

\caption
{The Gray Grid: the points and lines of the Gray configuration, viewed as points and lines on the 
$3 \times 3$ integer grid.}
\label{fig:GrayGrid}
\end{center}
\end{figure}

The labels, separated into symmetry classes, are shown in Table \ref{tab:SymmetryClassLabels}. Note that going from $R_{ij}$ to $R_{(i+1)j}$ ($B, G$ respectively) 
corresponds to adding 111, and going from $R_{ij}$ to $R_{i(j+1)}$ corresponds to adding 210 (with index arithmetic and point arithmetic all happening mod 3). 

For the lines, adding an index also corresponds to adding 111 or 210 respectively, but ignoring the * component (that is, $*+i = *$ for $i = 0,1,2$). 
For example, $X_{11} + 210 = *21 + 210 = (*+2)(2+1)(1+0) = *01 = X_{12}$ and $Y_{10} + 111 = 2* 1+111 = (2+1)(*+1)(1+1) = 0* 2 = Y_{20}$.

\begin{table}[htp]
\caption{The elements of the Gray configuration, labeled according to symmetry class over $\mathbb{Z}_{3} \times \mathbb{Z}_{3}$}
\begin{center}
\begin{align*}
\textcolor{red}{
\begin{array}{c c c}
R_{00} = 000 & R_{01} = 210 & R_{02} = 120 \\ 
R_{10} = 111 & R_{11} = 021 & R_{12} = 201 \\
R_{20} = 222 & R_{21} = 102 & R_{22} = 012 
\end{array}
}
&\quad
\begin{array}{c c c}
X_{00} = *00 & X_{01} = *10 & X_{02} = *20 \\ 
X_{10} = *11 & X_{11} = *21 & X_{12} = *01 \\
X_{20} = *22 & X_{21} = *02 & X_{22} = *12 
\end{array}
\\
\textcolor{blue}{
\begin{array}{c c c}
B_{00} = 100 & B_{01} = 010 & B_{02} = 220 \\ 
B_{10} = 211 & B_{11} = 121 & B_{12} = 001 \\
B_{20} = 022 & B_{21} = 202 & B_{22} = 112 
\end{array}
}
&\quad
\textcolor{orange}{
\begin{array}{c c c}
Y_{00} = 1\!*\! 0 & Y_{01} = 0\!*\! 0 & Y_{02} = 2\!*\! 0 \\ 
Y_{10} = 2\!*\! 1 & Y_{11} = 1\!*\! 1 & Y_{12} = 0\!*\! 1 \\
Y_{20} = 0\!*\! 2 & Y_{21} = 2\!*\! 2 & Y_{22} = 1\!*\! 2 
\end{array}
}\\
\textcolor{green!60!black}{
\begin{array}{c c c}
G_{00} = 110 & G_{01} = 020 & G_{02} = 200 \\ 
G_{10} = 221 & G_{11} = 101 & G_{12} = 011 \\
G_{20} = 002 & G_{21} = 212 & G_{22} = 122 
\end{array}
}
&\quad
\textcolor{cyan}{
\begin{array}{c c c}
Z_{00} = 11* & Z_{01} = 02* & Z_{02} = 20* \\ 
Z_{10} = 22* & Z_{11} = 10* & Z_{12} = 01* \\
Z_{20} = 00* & Z_{21} = 21* & Z_{22} = 12* 
\end{array}
}
\end{align*}

\end{center}
\label{tab:SymmetryClassLabels}
\end{table}%

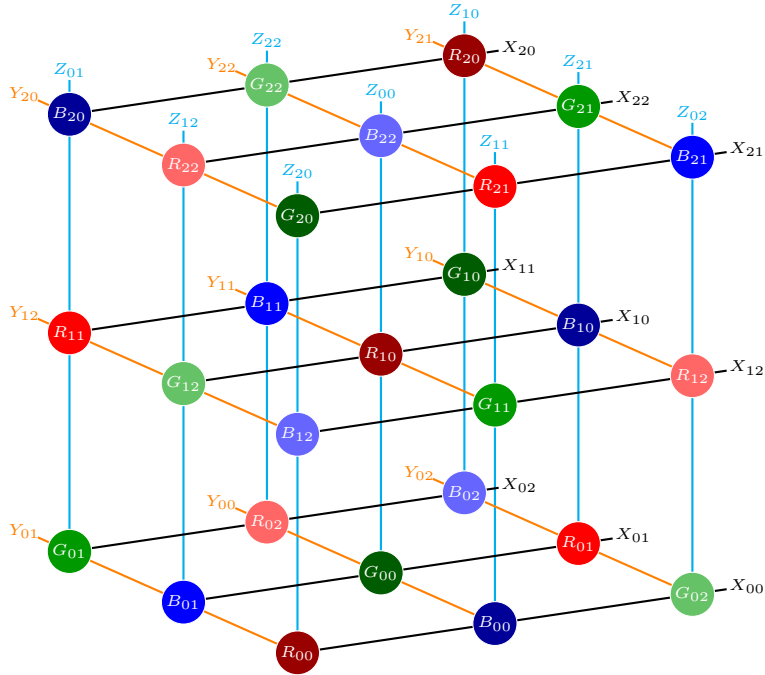
\begin{figure}[!h]
\begin{center}
\begin{tikzpicture}[vtx/.style={draw, circle, font = \tiny, fill = white, inner sep = 1pt},
  3d view,perspective
,
 ]  
\node[vtx,white, fill=red!60!black,] (R00) at ($3*(0,0,0)$) {$R_{00}$};
\node[vtx, white, fill=red!60!black] (R10) at ($3*(1,1,1)$) {$R_{10}$};
\node[vtx, white, fill=red!60!black] (R20) at ($3*(2,2,2)$) {$R_{20}$};
\node[vtx, white, fill=red] (R01) at ($3*(2,1,0)$) {$R_{01}$};
\node[vtx, white, fill=red] (R11) at ($3*(0,2,1)$) {$R_{11}$};
\node[vtx, white, fill=red] (R21) at ($3*(1,0,2)$) {$R_{21}$};
\node[vtx, white, fill=red!60!white] (R02) at ($3*(1,2,0)$) {$R_{02}$};
\node[vtx, white, fill=red!60!white] (R12) at ($3*(2,0,1)$) {$R_{12}$};
\node[vtx, white, fill=red!60!white] (R22) at ($3*(0,1,2)$) {$R_{22}$};

\node[vtx,white, fill=blue!60!black,] (B00) at ($3*(1,0,0)$) {$B_{00}$};
\node[vtx, white, fill=blue!60!black] (B10) at ($3*(2,1,1)$) {$B_{10}$};
\node[vtx, white, fill=blue!60!black] (B20) at ($3*(0,2,2)$) {$B_{20}$};
\node[vtx, white, fill=blue] (B01) at ($3*(0,1,0)$) {$B_{01}$};
\node[vtx, white, fill=blue] (B11) at ($3*(1,2,1)$) {$B_{11}$};
\node[vtx, white, fill=blue] (B21) at ($3*(2,0,2)$) {$B_{21}$};
\node[vtx, white, fill=blue!60!white] (B02) at ($3*(2,2,0)$) {$B_{02}$};
\node[vtx, white, fill=blue!60!white] (B12) at ($3*(0,0,1)$) {$B_{12}$};
\node[vtx, white, fill=blue!60!white] (B22) at ($3*(1,1,2)$) {$B_{22}$};

\node[vtx,white, fill=green!60!black!60!black,] (G00) at ($3*(1,1,0)$) {$G_{00}$};
\node[vtx, white, fill=green!60!black!60!black] (G10) at ($3*(2,2,1)$) {$G_{10}$};
\node[vtx, white, fill=green!60!black!60!black] (G20) at ($3*(0,0,2)$) {$G_{20}$};
\node[vtx, white, fill=green!60!black] (G01) at ($3*(0,2,0)$) {$G_{01}$};
\node[vtx, white, fill=green!60!black] (G11) at ($3*(1,0,1)$) {$G_{11}$};
\node[vtx, white, fill=green!60!black] (G21) at ($3*(2,1,2)$) {$G_{21}$};
\node[vtx, white, fill=green!60!black!60!white] (G02) at ($3*(2,0,0)$) {$G_{02}$};
\node[vtx, white, fill=green!60!black!60!white] (G12) at ($3*(0,1,1)$) {$G_{12}$};
\node[vtx, white, fill=green!60!black!60!white] (G22) at ($3*(1,2,2)$) {$G_{22}$};

\begin{scope}[on background layer]

\draw [thick,cyan,shorten >=-.75cm] (110) -- (112) node[font = \tiny, pos = 1.17] {$Z_{00}$};
\draw [thick,cyan,shorten >=-.75cm] (020) -- (022) node[font = \tiny, pos = 1.17] {$Z_{01}$};
\draw [thick,cyan,shorten >=-.75cm] (200) -- (202) node[font = \tiny, pos = 1.17] {$Z_{02}$};

\draw [thick,cyan,shorten >=-.75cm] (220) -- (222) node[font = \tiny, pos = 1.17] {$Z_{10}$};
\draw [thick,cyan,shorten >=-.75cm] (100) -- (102) node[font = \tiny, pos = 1.17] {$Z_{11}$};
\draw [thick,cyan,shorten >=-.75cm] (010) -- (012) node[font = \tiny, pos = 1.17] {$Z_{12}$};

\draw [thick,cyan,shorten >=-.75cm] (000) -- (002) node[font = \tiny, pos = 1.17] {$Z_{20}$};
\draw [thick,cyan,shorten >=-.75cm] (210) -- (212) node[font = \tiny, pos = 1.17] {$Z_{21}$};
\draw [thick,cyan,shorten >=-.75cm] (120) -- (122) node[font = \tiny, pos = 1.17] {$Z_{22}$};

\draw [thick, orange,shorten >=-.75cm] (100) -- (120) node[font = \tiny, pos = 1.35] {$Y_{00}$};
\draw [thick, orange,shorten >=-.75cm] (000) -- (020) node[font = \tiny, pos = 1.35] {$Y_{01}$};
\draw [thick, orange,shorten >=-.75cm] (200) -- (220) node[font = \tiny, pos = 1.35] {$Y_{02}$};

\draw [thick, orange,shorten >=-.75cm] (201) -- (221) node[font = \tiny, pos = 1.35] {$Y_{10}$};
\draw [thick, orange,shorten >=-.75cm] (101) -- (121) node[font = \tiny, pos = 1.35] {$Y_{11}$};
\draw [thick, orange,shorten >=-.75cm] (001) -- (021) node[font = \tiny, pos = 1.35] {$Y_{12}$};

\draw [thick, orange,shorten >=-.75cm] (002) -- (022) node[font = \tiny, pos = 1.35] {$Y_{20}$};
\draw [thick, orange,shorten >=-.75cm] (202) -- (222) node[font = \tiny, pos = 1.35] {$Y_{21}$};
\draw [thick, orange,shorten >=-.75cm] (102) -- (122) node[font = \tiny, pos = 1.35] {$Y_{22}$};

\draw [thick,shorten >=-.75cm] (000) -- (200) node[font = \tiny, pos = 1.22] {$X_{00}$};
\draw [thick,shorten >=-.75cm] (010) -- (210) node[font = \tiny, pos = 1.22] {$X_{01}$};
\draw [thick,shorten >=-.75cm] (020) -- (220) node[font = \tiny, pos = 1.22] {$X_{02}$};

\draw [thick,shorten >=-.75cm] (011) -- (211) node[font = \tiny, pos = 1.22] {$X_{10}$};
\draw [thick,shorten >=-.75cm] (021) -- (221) node[font = \tiny, pos = 1.22] {$X_{11}$};
\draw [thick,shorten >=-.75cm] (001) -- (201) node[font = \tiny, pos = 1.22] {$X_{12}$};

\draw [thick,shorten >=-.75cm] (022) -- (222) node[font = \tiny, pos = 1.22] {$X_{20}$};
\draw [thick,shorten >=-.75cm] (002) -- (202) node[font = \tiny, pos = 1.22] {$X_{21}$};
\draw [thick,shorten >=-.75cm] (012) -- (212) node[font = \tiny, pos = 1.22] {$X_{22}$};

\end{scope}

\end{tikzpicture}
\caption{The Gray Grid labeled with the symmetry classes from Table \ref{tab:SymmetryClassLabels}. }
\label{fig:GridWithZ3xZ3}
\end{center}
\end{figure}

Figure~\ref{fig:GridWithZ3xZ3} shows the Gray Grid using the labels from Table 
\ref{tab:SymmetryClassLabels}. This choice of labeling corresponds to the reduced Levi graph over 
$\mathbb{Z}_{3} \times \mathbb{Z}_{3}$ shown in Figure \ref{fig:Z3xZ3RLG} (see also Figure \ref{fig:K33}), where the first copy 
of $\mathbb{Z}_{3}$ corresponds to increasing the row index in Table 
\ref{tab:SymmetryClassLabels} (that is, to adding $+111$) and the second copy of $\mathbb{Z}_{3}
$ corresponds to increasing the column index (that is, to adding $+210$). 

In Figure~\ref{fig:Z3xZ3RLG},
the reduced Levi graph is oriented so that all arrows go from line classes to point classes, as mentioned 
above. Voltages are indicated as ordered pairs $(a,b) \in \mathbb{Z}_{3} \times \mathbb{Z}_{3}$, 
where $L \xrightarrow{(a,b)} P$ corresponds to an edge between $L_{ij}$ and $P_{(i+a)(j+b)}$ in 
the unreduced Levi graph, for $P \in \{R,G,B\}$, $L \in \{X,Y,Z\}$, $i,j \in \{0,1,2\}$. Unlabeled edges 
have voltage $(0,0)$.  Incrementing the first coordinate corresponds to increasing the row index 
(that is, to adding $+111$) and incrementing the second coordinate corresponds to increasing the column index (that is, to adding $+210$).

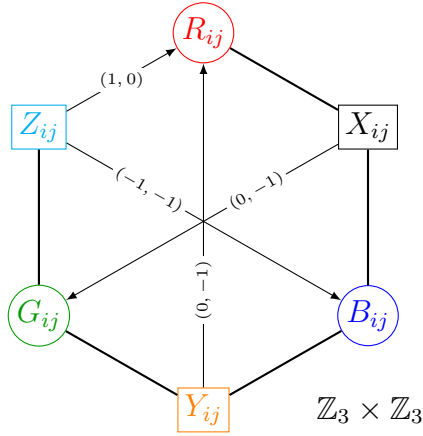
\begin{figure}[!h]
\begin{center}
\begin{tikzpicture}[vtx/.style={draw, circle, inner sep = 1 pt}, lin/.style = {draw, inner sep = 2.5 pt}, lbl/.style={midway, fill = white, inner sep = 1.5 pt, font = \tiny}]
\def\r{2.5}
\node[vtx, red] (R) at (90:\r) {$R_{ij}$};
\node[lin, cyan] (Z) at (90+60:\r) {$Z_{ij}$};
\node[vtx, green!60!black] (G) at  (90+2*60:\r) {$G_{ij}$};
\node[lin, , orange] (Y) at  (90+3*60:\r) {$Y_{ij}$};
\node[vtx, blue] (B) at  (90+4*60:\r) {$B_{ij}$};
\node[lin, ] (X) at  (90+5*60:\r) {$X_{ij}$};

\draw[thick] (X) -- (R);
\draw[thick] (X)--(B);
\draw[thick] (Y) -- (B);
\draw[thick] (Y) -- (G);
\draw[thick] (Z) -- (G);
\draw[-latex] (Z) -- node[lbl]{$(1,0)$} (R);
\draw[-latex] (Y) -- node[lbl, pos = .3, rotate = 90]{$(0,-1)$} (R);
\draw[-latex] (X) -- node[lbl, pos = .3, rotate = 30]{$(0,-1)$} (G);
\draw[-latex] (Z) -- node[lbl, pos = .3, rotate = -30]{$(-1,-1)$} (B);

\node[right = of Y] {$\mathbb{Z}_{3}\times \mathbb{Z}_{3}$};
\end{tikzpicture}
\caption{The reduced Levi graph for the Gray configuration with voltage group $\mathbb{Z}_{3} \times \mathbb{Z}_{3}$. }
\label{fig:Z3xZ3RLG}
\end{center}
\end{figure}

Expanding the second copy of $\mathbb{Z}_{3}$ (that is, having symmetry classes 
$\Box_{a0}$, $\Box_{a1}$, $\Box_{a2}$, $a \in \mathbb{Z}_{3}$, $\Box \in \{X,Y,Z,R,G,B\}$) 
gives a quotient whose underlying graph is isomorphic to the Pappus graph (see Figure \ref{fig:Pappus}).

\begin{figure}[htbp]
\begin{center}

\begin{tikzpicture}[vtx/.style={draw, circle, inner sep = 1 pt, font = \tiny}, lin/.style = {draw, inner sep = 2.5 pt, font = \tiny}, lbl/.style={midway, fill = white, 
inner sep = 1.5 pt, font = \tiny}]
\def\r{1.5}
\def\s{1}
\foreach\i in {0,1,2}{  
\node[vtx, red]  (R\i)  at ($(90:\r+\s*\i)$) {$R_{ j \i }$};
\node[lin, cyan] (Z\i) at (90+60:\r+\s*\i) {$Z_{j \i }$};
\node[vtx, green!60!black] (G\i) at  (90+2*60:\r+\s*\i) {$G_{j \i }$};
\node[lin, , orange] (Y\i) at  (90+3*60:\r+\s*\i) {$Y_{j \i }$};
\node[vtx, blue] (B\i) at  (90+4*60:\r+\s*\i) {$B_{j \i }$};
\node[lin, ] (X\i) at  (90+5*60:\r+\s*\i) {$X_{ j \i }$};
}

\foreach \i in {0,1,2}{
\draw[thick] (X\i) -- (R\i);
\draw[thick] (X\i)--(B\i);
\draw[thick] (Y\i) -- (B\i);
\draw[thick] (Y\i) -- (G\i);
\draw[thick] (Z\i) -- (G\i);
\draw[-latex]  (Z\i) -- node[lbl]{$1$} (R\i);
\draw[-latex] let \n1 = {int(mod(\i+2,3))} in (Y\i) to[bend left = 30] (R\n1);
\draw[-latex] let \n1 = {int(mod(\i+2,3))} in (X\i) to[bend left = 30]  (G\n1);
\draw[-latex]let \n1 = {int(mod(\i+2,3))} in (Z\i) to[bend left = 30] node[lbl, pos = .9, rotate = 0]{$-1$} (B\n1);
}
\node[right =  of Y2] {$\mathbb{Z}_{3}$};
\end{tikzpicture}
$\cong$ \hfill
\begin{tikzpicture}[vtx/.style={draw, circle, inner sep = 1 pt, font = \tiny}, lin/.style = {draw, inner sep = 2.5 pt, font = \tiny}, lbl/.style={above, fill = white, 
inner sep = 1.5 pt, font = \tiny}]
\def\r{1.5}
\def\s{1}
\foreach\i in {0,1,2}{  
\node[vtx, red]  (R\i)  at ($(90+\i*180:\r+\s*\i)$) {$R_{j \i }$};
\node[lin, cyan] (Z\i) at (90+60+\i*180:\r+\s*\i) {$Z_{j \i }$};
\node[vtx, green!60!black] (G\i) at  (90+2*60+\i*180:\r+\s*\i) {$G_{j \i }$};
\node[lin, , orange] (Y\i) at  (90+3*60+\i*180:\r+\s*\i) {$Y_{j \i }$};
\node[vtx, blue] (B\i) at  (90+4*60+\i*180:\r+\s*\i) {$B_{j \i }$};
\node[lin, ] (X\i) at  (90+5*60+\i*180:\r+\s*\i) {$X_{j \i }$};
}

\foreach \i in {0,1,2}{
\draw[thick] (X\i) -- (R\i);
\draw[thick] (X\i)--(B\i);
\draw[thick] (Y\i) -- (B\i);
\draw[thick] (Y\i) -- (G\i);
\draw[thick] (Z\i) -- (G\i);
\draw[-latex]  (Z\i) -- node[lbl]{$1$} (R\i);
\draw[-latex] let \n1 = {int(mod(\i+2,3))} in (Y\i) to[bend left = 30] (R\n1);
\draw[-latex] let \n1 = {int(mod(\i+2,3))} in (X\i) to[bend left = 30]  (G\n1);
\draw[-latex]let \n1 = {int(mod(\i+2,3))} in (Z\i) to[bend left = 30] node[lbl, pos = .9, rotate = 0]{$-1$} (B\n1);
}
\node[right =  of Y2] {$\mathbb{Z}_{3}$};
\end{tikzpicture}
\caption{Expanding the second copy of $\mathbb{Z}_{3}$ (that is, having symmetry classes 
$\Box_{a0}$, $\Box_{a1}$, $\Box_{a2}$, $a \in \mathbb{Z}_{3}$, $\Box \in \{X,Y,Z,R,G,B\}$) 
gives a quotient whose underlying graph is isomorphic to the Pappus graph. 
Two realizations are shown, one where the expansion is obvious and one where
 the underlying graph uses a more standard realization of the Pappus graph.}
\label{fig:PappusRLG}
\end{center}
\end{figure}
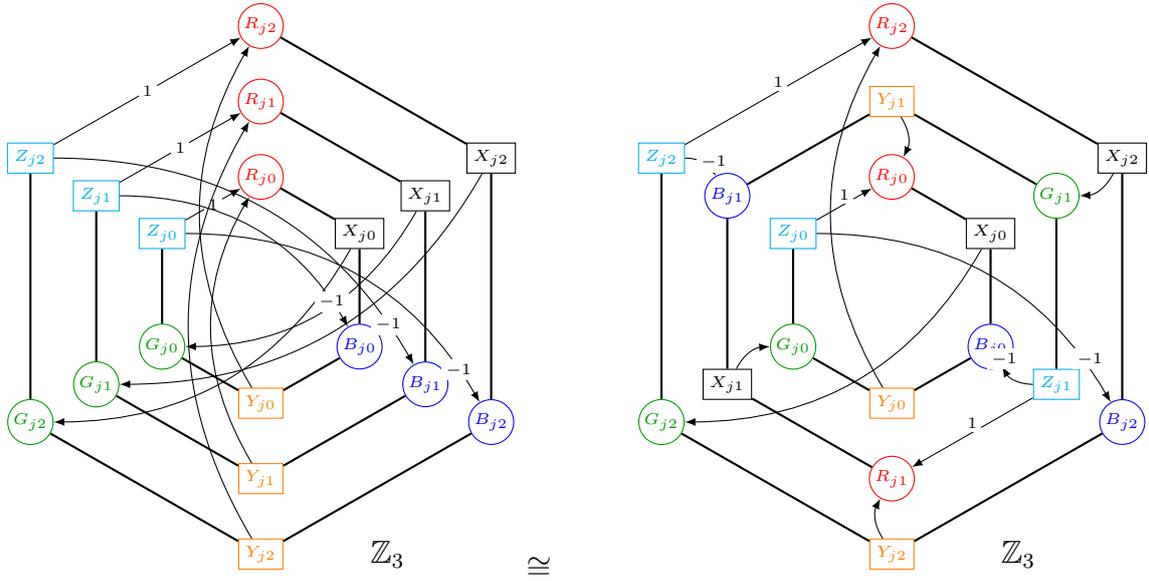

Expanding the first copy of $\mathbb{Z}_{3}$ (that is, having symmetry classes 
$\Box_{0j}$, $\Box_{1j}$, $\Box_{2j}$, $j \in \mathbb{Z}_{3}$, $\Box \in \{X,Y,Z,R,G,B\}$) 
gives a quotient isomorphic to the graph $GG$ shown in Figure \ref{fig:GG}. 

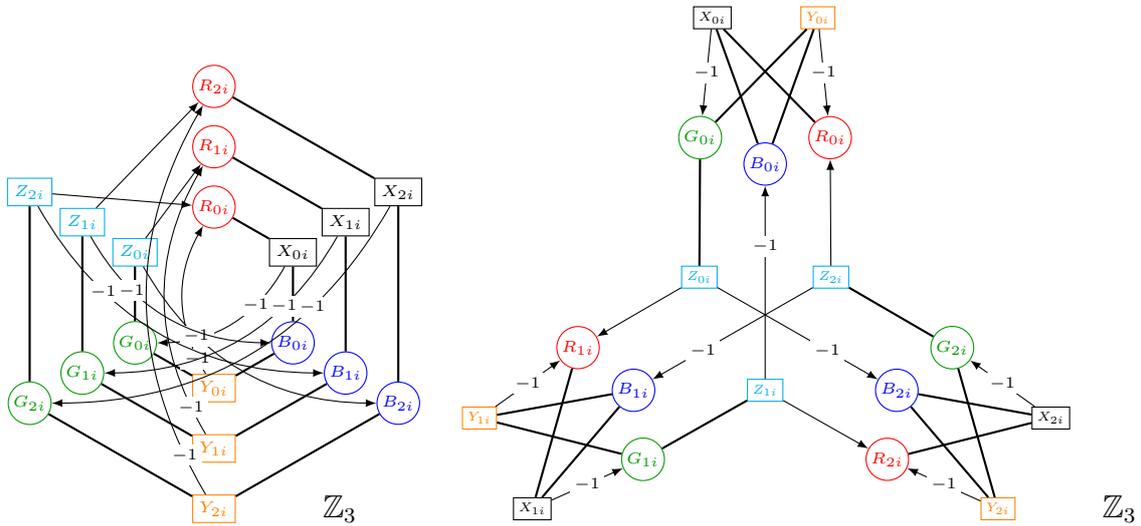
\begin{figure}[htbp]
\begin{center}

\begin{tikzpicture}[vtx/.style={draw, circle, inner sep = 1 pt, font = \tiny}, lin/.style = {draw, inner sep = 2.5 pt, font = \tiny}, lbl/.style={midway, fill = white, 
inner sep = 1.5 pt, font = \tiny}, scale = .8]
\def\r{1.5}
\def\s{1}
\foreach \j in {0,1,2}{  
\node[vtx, red]  (R\j)  at ($(90:\r+\s*\j)$) {$R_{ \j i}$};
\node[lin, cyan] (Z\j) at (90+60:\r+\s*\j) {$Z_{ \j i}$};
\node[vtx, green!60!black] (G\j) at  (90+2*60:\r+\s*\j) {$G_{ \j i}$};
\node[lin, , orange] (Y\j) at  (90+3*60:\r+\s*\j) {$Y_{ \j i}$};
\node[vtx, blue] (B\j) at  (90+4*60:\r+\s*\j) {$B_{ \j i}$};
\node[lin, ] (X\j) at  (90+5*60:\r+\s*\j) {$X_{ \j i}$};
}

\foreach \j in {0,1,2}{
\draw[thick] (X\j) -- (R\j);
\draw[thick] (X\j)--(B\j);
\draw[thick] (Y\j) -- (B\j);
\draw[thick] (Y\j) -- (G\j);
\draw[thick] (Z\j) -- (G\j);

\draw[-latex] let \n1 = {int(mod(\j+1,3))} in (Z\j) --  (R\n1);
\draw[-latex] (Y\j) to[bend left = 30] node[lbl, pos = .1, ]{$-1$} (R\j);

\draw[-latex] (X\j) to[bend left = 30] node[lbl, pos = .3]{$-1$} (G\j);

\draw[-latex] let \n1 = {int(mod(\j+2,3))} in (Z\j) to[bend right = 30] node[lbl, pos = .3]{$-1$} (B\n1);

}
\node[right =  of Y2] {$\mathbb{Z}_{3}$};
\end{tikzpicture}
\hfill
\begin{tikzpicture}[vtx/.style={draw, circle, inner sep = 1 pt, font = \tiny}, lin/.style = {draw, inner sep = 2.5 pt, font = \tiny, scale = .8}, lbl/.style={midway, fill = white, 
inner sep = 1.5 pt, font = \tiny}]
\def\r{2.5}
\def\rr{1}
\def\a{20}
\def\s{4}
\def\ss{2}
\foreach \j in {0,1,2}{  
\node[vtx, red]  (R\j)  at ($(90+\j*120 - \a:\r)$) {$R_{ \j i}$};
\node[lin, cyan] (Z\j) at (90+\j*120+60:\rr) {$Z_{ \j i}$};
\node[vtx, green!60!black] (G\j) at  (90+\j*120 + \a:\r) {$G_{ \j i}$};
\node[lin, , orange] (Y\j) at  (90+\j*120 - \a/2:\s) {$Y_{ \j i}$};
\node[vtx, blue] (B\j) at  (90+\j*120:\ss) {$B_{ \j i}$};
\node[lin, ] (X\j) at  (90+\j*120 + \a/2:\s) {$X_{ \j i}$};
}

\foreach \j in {0,1,2}{
\draw[thick] (X\j) -- (R\j);
\draw[thick] (X\j)--(B\j);
\draw[thick] (Y\j) -- (B\j);
\draw[thick] (Y\j) -- (G\j);
\draw[thick] (Z\j) -- (G\j);

\draw[-latex] let \n1 = {int(mod(\j+1,3))} in (Z\j) --  (R\n1);
\draw[-latex] (Y\j) to[bend right = 0] node[lbl, pos = .5, ]{$-1$} (R\j);

\draw[-latex] (X\j) to[bend left = 0] node[lbl, pos = .5]{$-1$} (G\j);

\draw[-latex] let \n1 = {int(mod(\j+2,3))} in (Z\j) to[bend right = 0] node[lbl, pos = .7]{$-1$} (B\n1);

}
\node[right =  of Y2] {$\mathbb{Z}_{3}$};
\end{tikzpicture}

\caption{Expanding the first copy of $\mathbb{Z}_{3}$ (that is, having symmetry classes 
$\Box_{0j}$, $\Box_{1j}$, $\Box_{2j}$, $j \in \mathbb{Z}_{3}$) gives a quotient isomorphic to the graph $GG$. 
Again, two drawings are shown, one corresponding directly to the expansion and one obviously isomorphic 
to the drawing of $GG$ shown in Figure \ref{fig:GG}.}
\label{fig:GG-RLG}
\end{center}
\end{figure}


\section{A polycyclic realization of the Gray Configuration with the Pappus RLG}  \label{sect: Gray_with_Pappus}


In this section, we show that the $\mathbb{Z}_{3}$ quotient graph, which as an unlabeled graph, 
is isomorphic to the Pappus graph (Figure \ref{fig:Pappus}),  produces a polycyclic realization of 
the Gray graph with $\mathbb{Z}_{3}$ symmetry. Consider the re-drawing of the Pappus voltage 
graph shown in Figure~\ref{fig:PapOnHC}, in which a particular Hamiltonian cycle is chosen to 
be on the boundary of the graph.

\begin{figure}[!h]
\begin{center}
\begin{tikzpicture}[vtx/.style={draw, circle, inner sep = 1 pt, font = \tiny}, lin/.style = {draw, inner sep = 2.5 pt, font = \tiny}, lbl/.style={midway, fill = white, inner sep = 1.5 pt, font = \tiny}, scale = .65]
\def\r{4.5}
\def\s{1}

\node[vtx, red] (R0) at (360/18*0:\r){$R_{j0}$};
\node[lin] (X0) at (360/18*1:\r){$X_{j0}$};
\node[vtx, green!60!black] (G2) at (360/18*2:\r){$G_{j2}$};
\node[lin, cyan] (Z2) at (360/18*3:\r){$Z_{j2}$};
\node[vtx, blue] (B1) at (360/18*4:\r){$B_{j1}$};
\node[lin] (X1) at (360/18*5:\r){$X_{j1}$};
\node[vtx, green!60!black] (G0) at (360/18*6:\r){$G_{j0}$};
\node[lin, cyan] (Z0) at (360/18*7:\r){$Z_{j0}$};
\node[vtx, blue] (B2) at (360/18*8:\r){$B_{j2}$};
\node[lin, orange] (Y2) at (360/18*9:\r){$Y_{j2}$};
\node[vtx, red] (R1) at (360/18*10:\r){$R_{j1}$};
\node[lin, cyan] (Z1) at (360/18*11:\r){$Z_{j1}$};
\node[vtx, blue] (B0) at (360/18*12:\r){$B_{j0}$};
\node[lin, orange] (Y0) at (360/18*13:\r){$Y_{j0}$};
\node[vtx, red] (R2) at (360/18*14:\r){$R_{j2}$};
\node[lin, ] (X2) at (360/18*15:\r){$X_{j2}$};
\node[vtx, green!60!black] (G1) at (360/18*16:\r){$G_{j1}$};
\node[lin, orange] (Y1) at (360/18*17:\r){$Y_{j1}$};

\path (Z2.north west) node[above right, font=\tiny] {$\stackrel{\text{start}}{\downarrow}$};


\foreach \i in {0,1,2}{
\draw[thick] (X\i) -- (R\i);
\draw[thick] (X\i)--(B\i);
\draw[thick] (Y\i) -- (B\i);
\draw[thick] (Y\i) -- (G\i);
\draw[thick] (Z\i) -- (G\i);
\draw[-latex]  (Z\i) -- node[lbl]{$1$} (R\i);
\draw[-latex] let \n1 = {int(mod(\i+2,3))} in (Y\i) to[bend left = 0] (R\n1);
\draw[-latex] let \n1 = {int(mod(\i+2,3))} in (X\i) to[bend left = 0]  (G\n1);
\draw[-latex]let \n1 = {int(mod(\i+2,3))} in (Z\i) to[bend left = 0] node[lbl, , rotate = 0]{$-1$} (B\n1);
}
\node[right =  of X2] {$\mathbb{Z}_{3}$};
\end{tikzpicture}
\hfill
 \begin{tikzpicture}[vtx/.style={draw, circle, inner sep = 1 pt, font = \tiny}, lin/.style = {draw, inner sep = 2.5 pt, font = \tiny}, lbl/.style={midway, fill = white, inner sep = 1.5 pt, font = \tiny}, scale = .65]
\def\r{4.5}
\def\s{1}

\node[vtx, red] (R0) at (360/18*0:\r){$R_{j0}$};
\node[lin] (X0) at (360/18*1:\r){$X_{j0}$};
\node[vtx, green!60!black] (G2) at (360/18*2:\r){$G_{j2}$};
\node[lin, cyan] (Z2) at (360/18*3:\r){$Z_{j2}$};
\node[vtx, blue] (B1) at (360/18*4:\r){$B_{j1}$};
\node[lin] (X1) at (360/18*5:\r){$X_{j1}$};
\node[vtx, green!60!black] (G0) at (360/18*6:\r){$G_{j0}$};
\node[lin, cyan] (Z0) at (360/18*7:\r){$Z_{j0}$};
\node[vtx, blue] (B2) at (360/18*8:\r){$B_{j2}$};
\node[lin, orange] (Y2) at (360/18*9:\r){$Y_{j2}$};
\node[vtx, red] (R1) at (360/18*10:\r){$R_{j1}$};
\node[lin, cyan] (Z1) at (360/18*11:\r){$Z_{j1}$};
\node[vtx, blue] (B0) at (360/18*12:\r){$B_{j0}$};
\node[lin, orange] (Y0) at (360/18*13:\r){$Y_{j0}$};
\node[vtx, red] (R2) at (360/18*14:\r){$R_{j2}$};
\node[lin, ] (X2) at (360/18*15:\r){$X_{j2}$};
\node[vtx, green!60!black] (G1) at (360/18*16:\r){$G_{j1}$};
\node[lin, orange] (Y1) at (360/18*17:\r){$Y_{j1}$};

\path (B2.west) node[left, font=\tiny] {start $\to$};


\foreach \f/\g in {X0/R0, Y1/R0, Y1/G1, X2/R2, R2/Y0,Y0/B0, Z1/B0,Z1/R1, Y2/R1, Y2/B2, Z0/G0, X1/G0, X1/R1, X0/B0,X1/B1,Z2/B1,Z2/G2,X0/G2,X2/G1,Z0/R0}{\draw[ thick] (\f) -- (\g);}

\foreach \f/\g/\h in {Z0/B2/{-1},Y2/G2/{-1}, Y1/B1/1, X2/B2/1,Z2/R2/1, Y0/G0/1,Z1/G1/1}{\draw[,-latex ] (\f) -- node[lbl]{$\h$} (\g);}

\begin{scope}[on background layer]
\foreach \f/\g in {B2/Y2, Y2/R1, R1/Z1, Z1/B0, B0/Y0, Y0/R2,G1/Y1, Y1/R0}{\draw[yellow, line width = 1.5mm, opacity = .5] (\f) -- (\g);}
\foreach \f/\g in {X2/R2,X2/B2, X2/G1, Z1/G1,X0/R0,X0/B0,G2/X0,G2/Y2, Z2/G2,Z2/R2,B1/Z2, B1/Y1, X1/B1, X1/R1, G0/X1, G0/Y0, Z0/G0,Z0/R0}{\draw[green, line width = 1mm, opacity = .5] (\f) -- (\g);}
\foreach \f/\g in {Z0/B2}{\draw[cyan, line width =2mm, opacity = .5] (\f) -- (\g);}
\end{scope}

\node[right =  of X2] {$\mathbb{Z}_{3}$};
\end{tikzpicture}

\caption{(Left) A drawing of the Pappus RLG using a Hamiltonian cycle on the boundary, used in 
producing a polycyclic realization of the Gray configuration whose reduced Levi graph is the 
Pappus RLG. (Right) Relabeling to zero out the boundary edges. In the construction of a 
corresponding polycyclic realization, yellow highlights correspond to incidences with one degree of 
freedom, green highlights to incidences that are determined, and the blue highlights to a final 
incidence that results from a continuity argument. }
\label{fig:PapOnHC}
\end{center}
\end{figure}
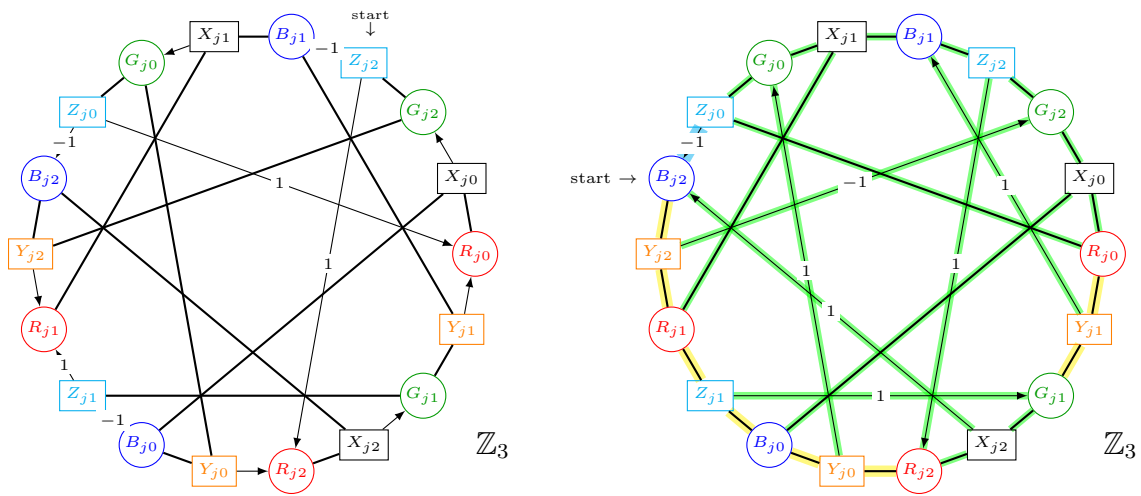

Two useful facts about voltage graphs (which reduced Levi graphs are), are the following: (1) 
Adding an element of the voltage group to all edges incident with a node in the voltage graph 
results in an isomorphic lift graph (roughly, it corresponds to changing the designation of the ``0-th'' 
element of a particular symmetry class), and (2) consequently, given any spanning tree in a voltage 
graph, it is possible to zero-out the labels on that spanning tree. See \cite{PisSer2013} for more 
detailed information, and \cite{Ber2013}, especially Figure 6, for a worked out example.
Beginning with the Pappus graph drawn with a Hamiltonian cycle on the boundary shown in Figure 
\ref{fig:PapOnHC}, we add and subtract voltages as necessary around the perimeter to zero-out all 
but one of the voltages on the outside Hamiltonian cycle, which will make construction of the 
corresponding realization of the Gray configuration more tractable.  (Specifically, we started, 
somewhat arbitrarily at the node $Z_{j2}$ and added +1 to each of the labels on the incident edges, 
which turned  $Z_{j2} \xrightarrow{-1} B_{j1}$ into $Z_{j2} \to B_{j1}$ while also turning  $Z_{j2} 
\xrightarrow{1} R_{j2}$ to $Z_{j2} \xrightarrow{2} R_{j2}$ and $Z_{j2} \to G_{j2}$ to $Z_{j2} 
\xrightarrow{1} G_{j2}$. Next, we added $-1$ to all the edges incident with $G_{j2}$, which zeroes 
out the previously assigned $Z_{j2} \xrightarrow{1} G_{j2}$ while modifying the other two edges, 
including the one which becomes $G_{j2} \xrightarrow{-1} X_{j0}$, which is zeroed out by adding 
+1 to $X_{j0}$, and so on around the boundary, until all but one of the edges in the boundary cycle 
has label 0.)


To construct a configuration with $\mathbb{Z}_{3}$ symmetry with this reduced Levi graph, we follow the same sort of construction techniques that were outlined 
in~\cite{Ber2013}, beginning at point symmetry class $B_{j2}$ and proceeding counterclockwise around the boundary of the RLG. At each step except the last, we 
are doing one of the following:

\begin{itemize}
\item Initialization: construct the point class $B_{j2}$ as the vertices of a regular 3-gon centered at $\mc{O}$: specifically, let $B_{j2} = (\cos(2 \pi j/3), \sin(2 \pi j/3))$.
\item Draw a line (class) arbitrarily through a point (class) (one degree of freedom, denoted by a yellow highlight on the corresponding edge);
\item Place a point (class) arbitrarily on a previously drawn line class (one degree of freedom, denoted by a yellow highlight on the corresponding edge);
\item Construct a line (class) as the join of two points (green highlight on the corresponding graph edges)
\item Construct a point (class) as the meet of two lines (green highlight on the corresponding graph edges)
\end{itemize}
Each of these steps depends on at most two previously-constructed elements. 

In the final step, according to the constructions in the RLG, we need to have the line $Z_{j0}$ be incident with point $G_{j0}$, $R_{j0}$ and $B_{(j-1)2}$ 
(cyan highlight on the corresponding graph edges, which we accomplish via a continuity argument.

\newcommand{\meet}[2]{\textbf{meet}(#1, #2)}
\newcommand{\join}[2]{\textbf{join}(#1, #2)}

\newpage

The specific construction steps are as follows:
\begin{enumerate}
\item Construct $B_{j2} = (\cos(2 \pi j/3), \sin(2 \pi j/3))$.
\item Construct line $Y_{02}$ arbitrarily through $B_{02}$, and construct $Y_{j2}$ by rotating $Y_{02}$ by  $2 \pi j/3$ about $(0,0)$ (henceforth, ``by rotation'').
\item Construct $R_{01}$ arbitrarily on $Y_{02}$ and the rest of the $R_{j1}$ by rotation.
\item Construct $Z_{01}$ arbitrarily through $R_{01}$, and the rest of the $Z_{j1}$ by rotation.
\item Construct $B_{00}$ arbitrarily on $Z_{01}$ and the rest of the $B_{j0}$ by rotation.
\item Construct $Y_{00}$ arbitrarily through $B_{00}$, and the rest of the $Y_{j0}$ by rotation.
\item Construct $R_{02}$ arbitrarily on $Y_{00}$ and the rest of the $R_{j2}$ by rotation.
\item Construct $X_{02} = \join{R_{02}}{B_{12}}$ (corresponding to the label $X_{j2}\xrightarrow{1} B_{j2})$ and the rest of the $X_{j2}$ by rotation.
\item Construct $G_{01} = \meet{X_{02}}{Z_{21}}$ and the rest of the $G_{j2}$ by rotation. Note that the arrow $Z_{j1}\xrightarrow{1} G_{j1}$ says that for each $j$, $Z_{j1} 
\sim G_{(j+1)1}$, or alternately $Z_{(j-1)1} \sim G_{j1}$, and $-1 \equiv 2 \bmod 3$. 
\item Construct $Y_{01}$ arbitrarily through $G_{01}$, and the rest of the $Y_{j1}$ by rotation. 
\item Construct $R_{00}$ arbitrarily on $Y_{01}$ and the rest of the $R_{j0}$ by rotation.
\item Construct $X_{00} = \join{R_{00}}{B_{00}}$ and the rest of the $X_{j0}$ by rotation.
\item Construct $G_{02} = \meet{X_{00}}{Y_{12}}$ (corresponding to the label $Y_{j2} \xrightarrow{-1} G_{j2}$) and the rest of the $G_{j2}$ by rotation.
\item Construct $Z_{02} = \join{G_{02}}{R_{12}}$ (corresponding to the label $Z_{j2} \xrightarrow{1} R_{j2}$) and the rest of the $Z_{j2}$ by rotation.
\item Construct $B_{01} = \meet{Z_{02}}{Y_{21}}$ (corresponding to the label $Y_{j1} \xrightarrow{-1} B_{j1}$) and the rest of the $B_{j1}$ by rotation.
\item Construct $X_{01} = \join{B_{01}}{R_{01}}$ and the rest of the $X_{j1}$ by rotation. 
\item Construct $G_{00} = \meet{X_{01}}{Y_{20}}$ (corresponding to the label $Y_{j0} \xrightarrow{1} G_{j0}$) and the rest of the $G_{j0}$ by rotation.
\item Finally, construct $Z_{00}=\join{G_{00}}{R_{00}}$ and the rest of the $Z_{j0}$ by rotation. 
\end{enumerate}

The missing incidence, indicated in cyan, is that line $Z_{00}$ needs to pass through $B_{22}$ (and by symmetry, $Z_{10} \sim B_{02}$, $Z_{20} \sim B_{12}$). 
This can be accomplished by a continuity argument, observing that, for example, moving the last point class $R_{j0}$ that has a degree of freedom sweeps 
the resulting line $Z_{00}$ across $B_{22}$ (and corresponding for the other two lines in the class). This is illustrated in the three snapshot constructions 
shown in Figure \ref{fig:PappusCfgRealization}.

\begin{figure}[htbp]
\begin{center}

\begin{subfigure}{.45\linewidth}
\includegraphics[width=\linewidth]{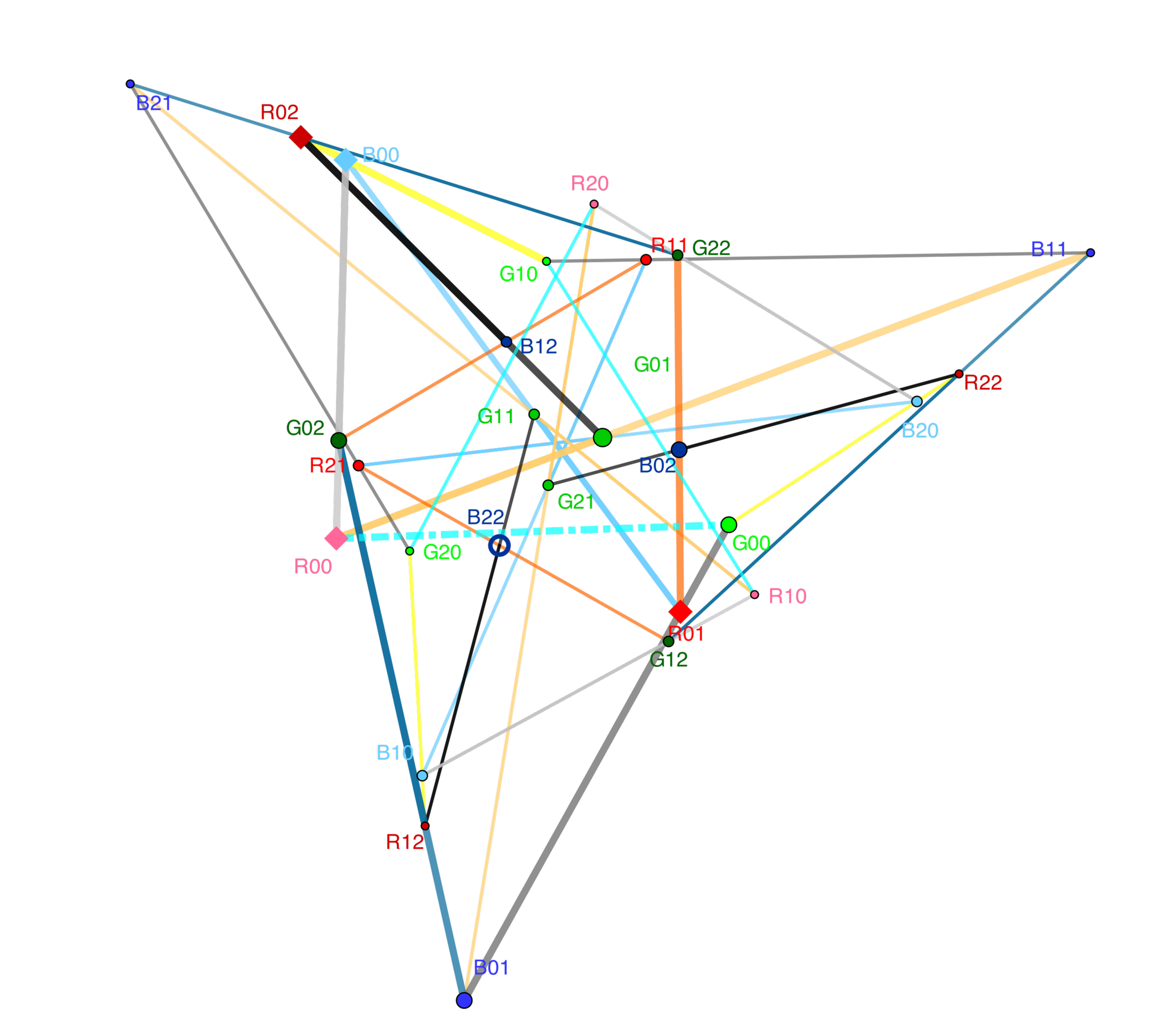}
\caption{Line $Z_{00}$ (dashed) is just above point $B_{22}$ (hollow)}
\end{subfigure}
\hfill
\begin{subfigure}{.45\linewidth}
\includegraphics[width=\linewidth]{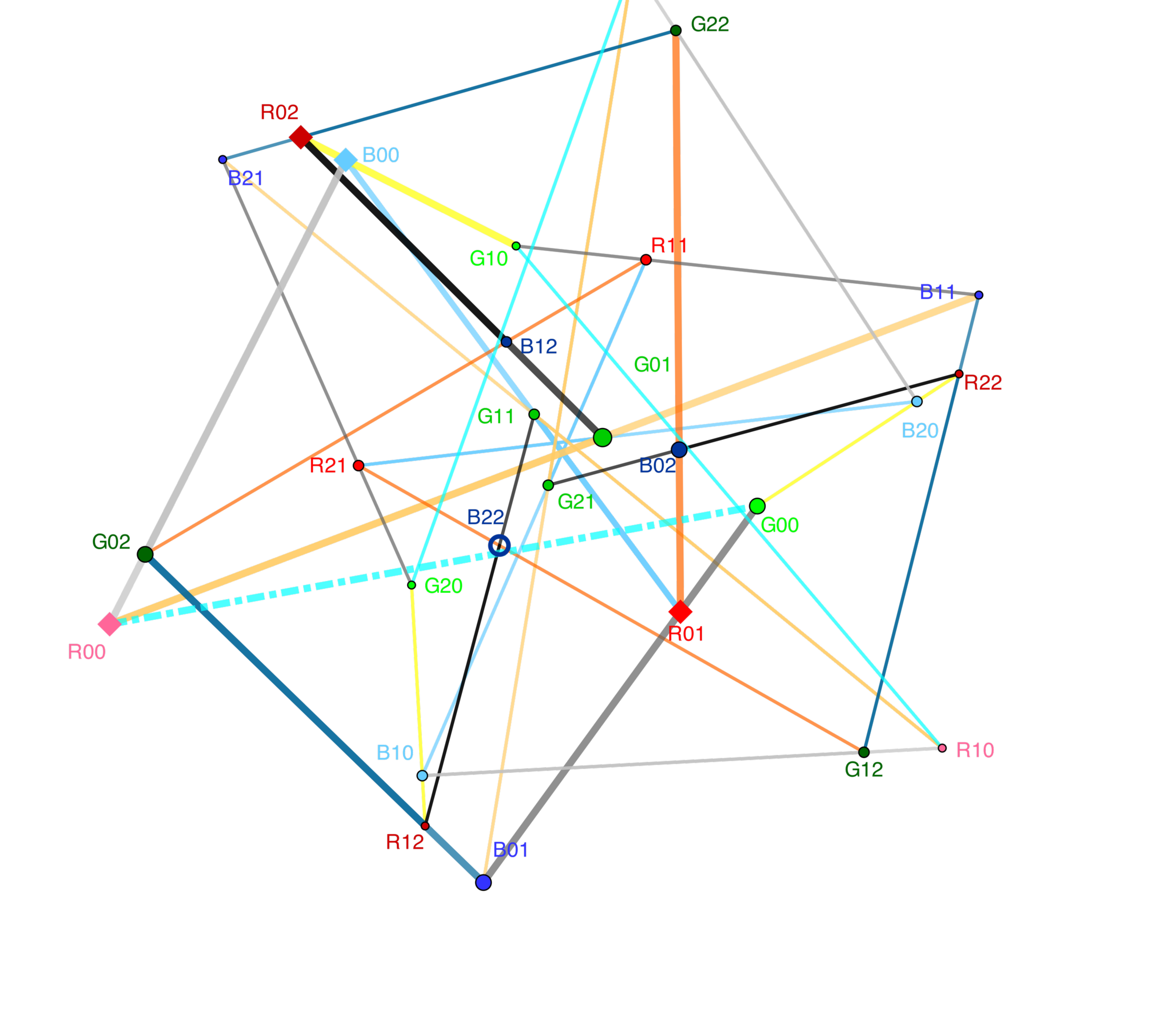}
\caption{Line $Z_{00}$ (dashed) is just below point $B_{22}$ (hollow)}
\end{subfigure}
\begin{subfigure}{.75\linewidth}
\includegraphics[width=\linewidth]{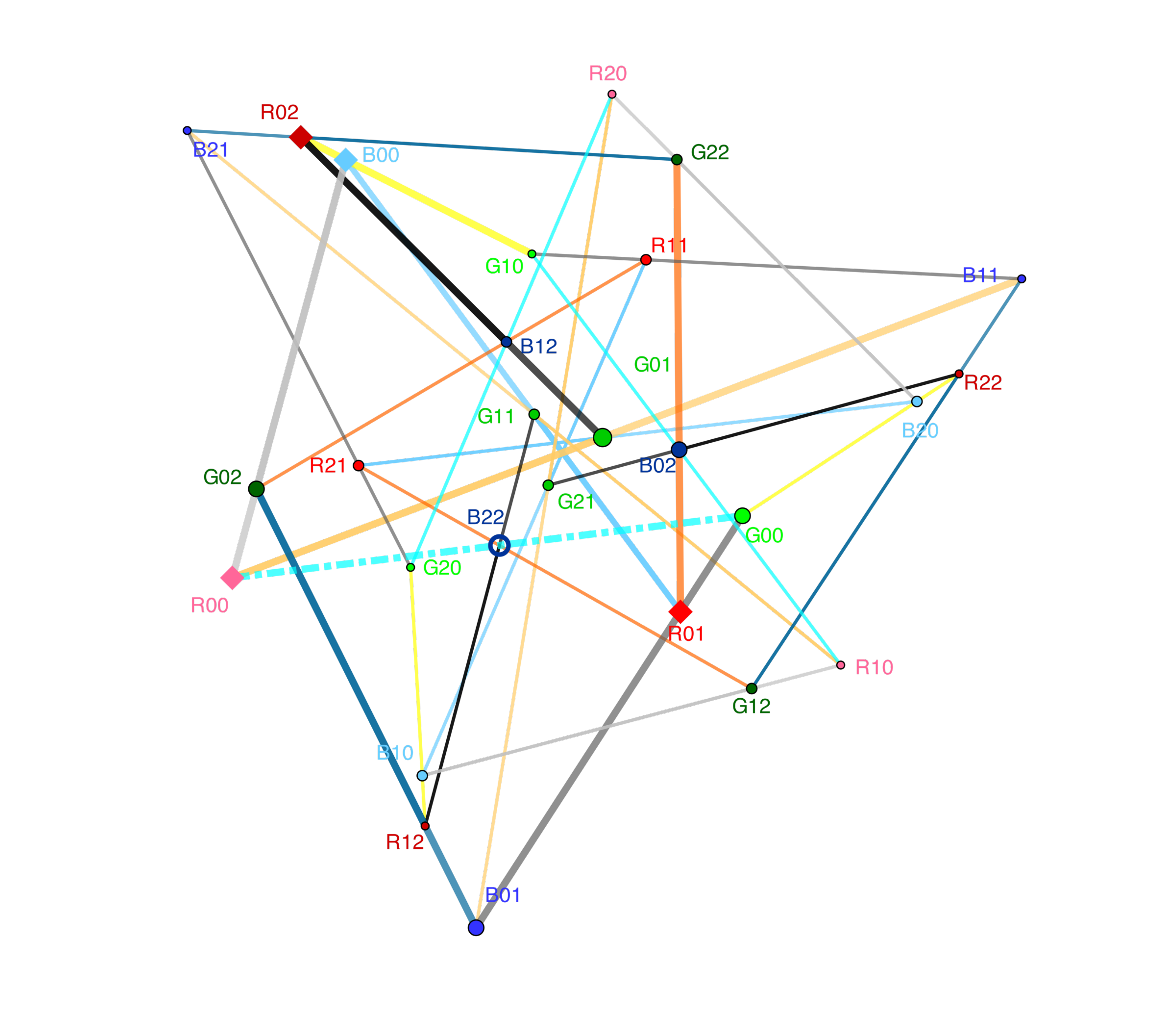}
\caption{The exact configuration, by continuity. }
\end{subfigure}

\caption{Realizing the Gray configuration with the Pappus RLG. The 0th element of each point and line class is shown larger/thicker. Points $R$, $G$, $B$ are shown
in shades of red, green, blue respectively (going from light to dark as $j = 0,1,2$) and similarly for line classes $X, Y, Z$ using shades of black, yellow, cyan. Points 
shown  with diamonds are movable. (The movable lines are not specifically indicated.)  Moving $R_{00}$ along its line moves the cyan dashed line $Z_{00}$ from 
above the point $B_{22}$ (shown hollow) to below the point $B_{22}$, so by continuity there is a position of $R_{22}$ in which line $Z_{00}$ passes through $B_{22}$ 
exactly. The $\mathbb{Z}_{3}$ action is counterclockwise rotation through $2 \pi/3$ and corresponds to adding $+111$ to each point label.}
\label{fig:PappusCfgRealization}
\end{center}
\end{figure}


\section{A polycyclic realization of the Gray configuration with threefold rotational symmetry,  using the $GG$ RLG.}


A polycyclic realization of the Gray configuration with threefold rotational symmetry is depicted in 
Figure~\ref{fig:GrayPoly}. The reduced Levi graph using $\mathbb{Z}_{3}$ as the voltage group is $GG$.

In what follows we explain how this realization is constructed.
\begin{figure}[!h]
\begin{center}
\includegraphics[width=0.875\textwidth]{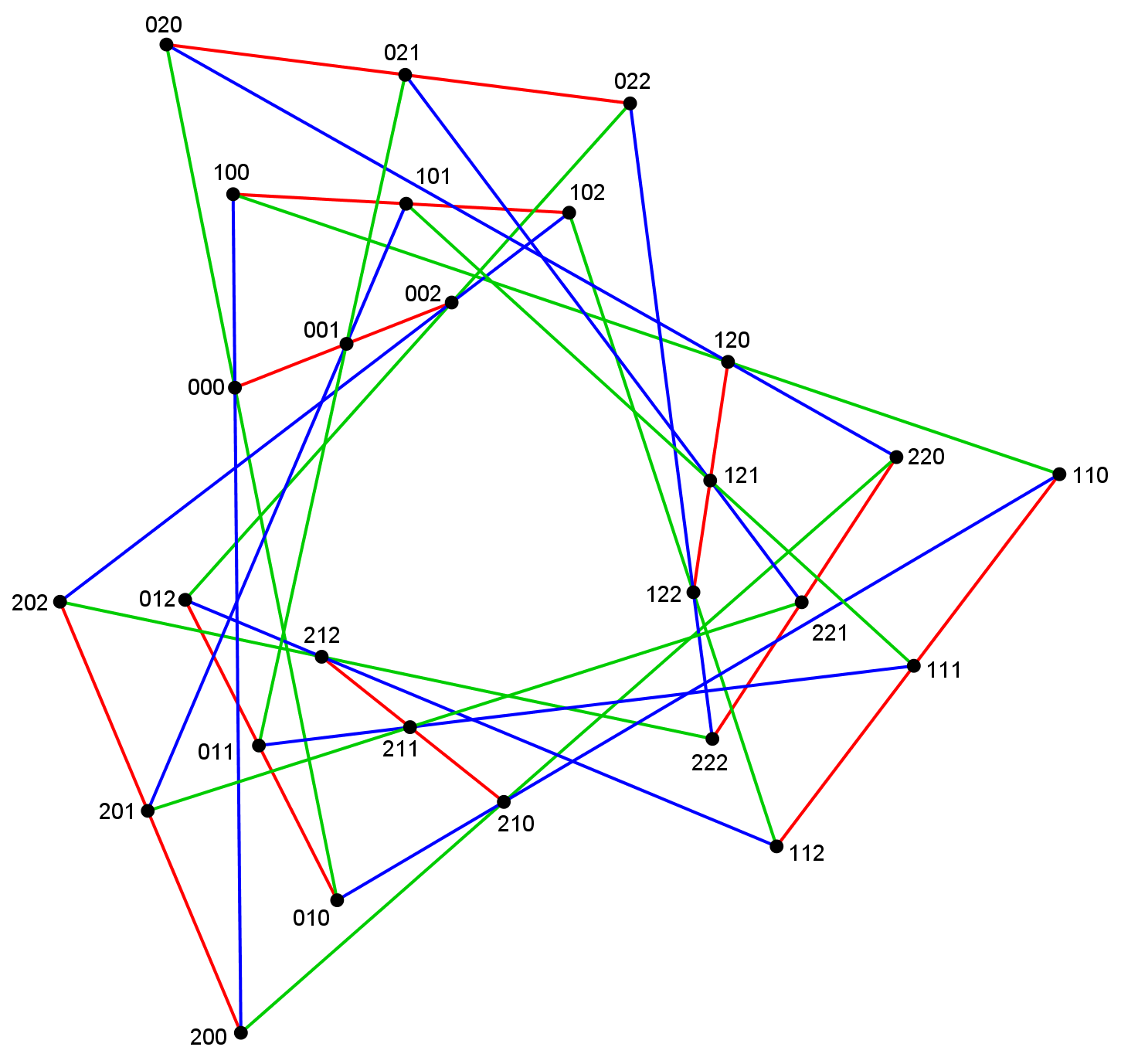}
\caption{The Gray configuration: a polycyclic realization with 3-fold rotational symmetry.
The labels of points are those introduced in Section 1. The colors em\-pha\-size resolvability of the configuration (show the parallel classes). 
Note that the $\mathbb{Z}_{3}$ action corresponds to adding $+210$ to each point label.}
\label{fig:GrayPoly}
\end{center}
\end{figure}

We start from a polycyclic realization of the Pappus configuration, see Figure~\ref{fig:PappusPoly}. This realization is well known; it occurs
e.g.\ in~\cite[Figure 1.16]{Gru2000} and in~\cite[Figure 1.10]{PisSer2013}. The Pappus configuration contains as a subconfiguration the 
$(9_2, 6_3)$ ``grid'' configuration, which is shown in Figure~\ref{fig:PappusPoly} by blue and green lines. The labels of the points in that 
figure verify that this is so, indeed (note that the third coordinate in these labels show that this configuration can be conceived as lying 
in the $XY$ plane of a spatial Cartesian coordinate system). This implies that using two additional suitable copies of the $(9_2, 6_3)$
configuration (along with adding 9 independent lines), one obtains a realization of the Gray configuration given in Figure~\ref{fig:GrayPoly}.
\begin{figure}[!h]
\begin{center}
\includegraphics[width=0.5\textwidth]{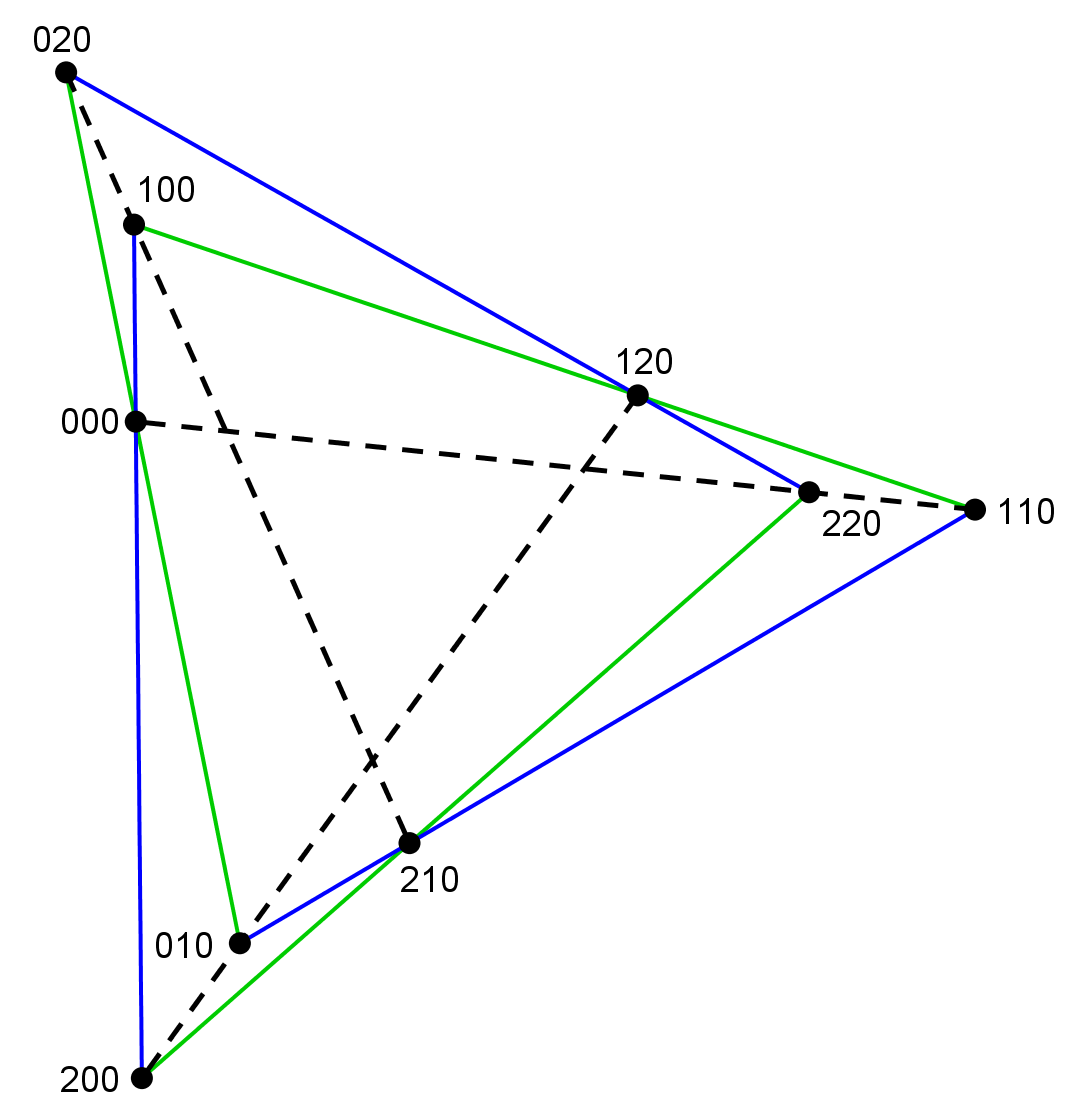}
\caption{Polycyclic realization of the Pappus configuration with 3-fold rotational symmetry.
Removal of the dashed lines gives rise to a polycyclic realization of the $(9_2, 6_3)$ ``grid'' configuration.}
\label{fig:PappusPoly}
\end{center}
\end{figure}

To this end we need the following Theorem~\cite{Mol1990, Yag1968}.

\begin{theorem} \label{thm:Yaglom}
Assume that a geometric figure $\mathcal F$ changes continuously in such a way that
\begin{enumerate}[\rm{(}1\rm{)}]
   \item precisely one of its points is fixed (denote it by $O$);
   \item it is at all times directly similar to its original copy.
\end{enumerate}
Consider two points $P, P'\in \mathcal F$, both different from $O$. Then, if $P$ moves along a path $\ell$, 
then $P'$ moves along a path $\ell'$ such that $\ell'$ is an image of $\ell$ under a dilative rotation. 
\end{theorem}
Note that ``directly similar'' means that while changing $\mathcal F$, its orientation is preserved; in this procedure, images of 
$\mathcal F$ occur under the action of a one-parameter family of dilative rotations (also known as \emph{spiral similarities}). 
For some properties of a \emph{dilative rotation}, see e.g.~\cite{CoxGre}.   

We apply this theorem in the following way. Take a copy of a polycyclic realization of the $(9_2, 6_3)$ ``grid'' configuration
(let it be denoted by $\mathcal G_0$). Denote its centre of rotation by $O$, and fix this point; it plays the role of the point 
$O$ of the theorem. Choose a straight line which passes through a configuration point of this grid, but avoids all its other 
configuration points as well as the centre $O$; this will play the role of the path $\ell$ of the theorem, thus we shall refer 
to it by the same notation. Considering our Figure~\ref{fig:GrayPoly}, the starting copy of the grid configuration can be 
taken as a copy of precisely what is depicted in Figure~\ref{fig:PappusPoly} (with the same labels of points). In addition, 
the line $\ell$ is taken as the red line through the point $(020)$.

Now take the copies $\mathcal G_1$ and $\mathcal G_2$ which are images of $\mathcal G_0$ under dilative rotations 
such that their points $(021)$ and  $(022)$, respectively correspond to the point $(020)$, and lie on the line $\ell$. As a 
consequence of Theorem~\ref{thm:Yaglom} above, we have that the points of the set $\mathcal G_0 \cup \mathcal G_1 
\cup \mathcal G_2$ are arranged into collinear triples along the 9 red lines of our Figure~\ref{fig:GrayPoly} (note that all 
these lines are copies of $\ell$ under dilative rotations, again due to the theorem). As a result, the set $\mathcal G_0 \cup 
\mathcal G_1 \cup \mathcal G_2$, together with the 9 new lines, forms a configuration which is isomorphic to the Gray 
configuration; moreover, it is polycyclic with threefold rotational symmetry.

\begin{figure}[htbp]
\begin{center}
\includegraphics[width = .66\linewidth]{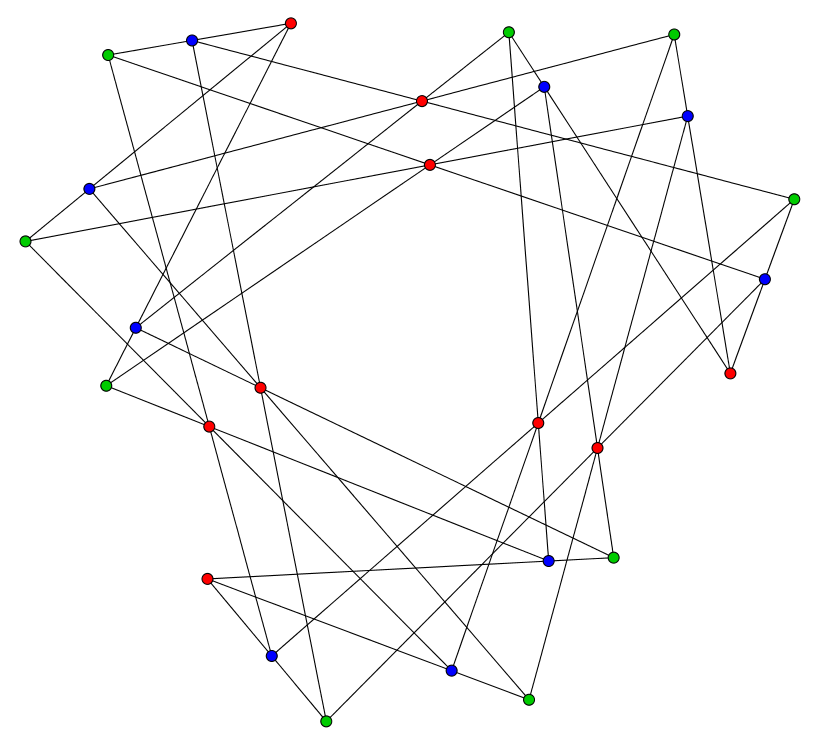}
\caption{Using reciprocation on the configuration shown in Figure \ref{fig:GrayPoly}, we also construct a realization of the dual configuration to the Gray configuration.}
\label{default}
\end{center}
\end{figure}

Clearly, the realization of the Gray configuration shown in Figure \ref{fig:GrayPoly} has 3-fold rotational symmetry, and it is easy to verify that the 
rotation corresponds to adding +210 to each of the point labels. We previously showed that adding +210 corresponds to the reduced Levi graph 
shown in Figure ~\ref{fig:GG-RLG}, so this realization is a polycylic realization with graph $GG$ as its reduced Levi graph.

This geometric realization has been used to construct a unit-distance realization of the Gray graph; see \cite{BerGevPis2025}.


\section{9-fold symmetry of the Gray Graph and Gray Configuration} \label{sect:9fold}


\subsection{9-fold symmetry of the Gray graph}

Figure~\ref{fig:GrayZ9} shows two drawings of the Gray Graph with 9-fold rotational symmetry, which 
interact nicely with the Pappus realization. (The graph on the left has the positions of the rings of 
points and lines chosen so that the graph is intelligible, while the graph on the right has the $GG$ 
symmetry class elements lined up; there is no change in the order of the elements along each 
rotational ring, just in the position of the 0th element of each ring of points and of lines.) Specifically, 
3-fold rotation preserves the symmetry classes under the Pappus action. For example, considering 
the class $B_{i0} = \{B_{00} = 100, B_{10} = 211, B_{20} = 022\}$ shown in blue, 
located on the outermost ring of the graph,  rotation by $120^{\circ}$ maps $B_{00} \to B_{10} 
\to B_{20}$. It is easy to verify that all of the Pappus symmetry classes (the color-coded columns in 
Table \ref{tab:SymmetryClassLabels}) are preserved by this 3-fold rotation.

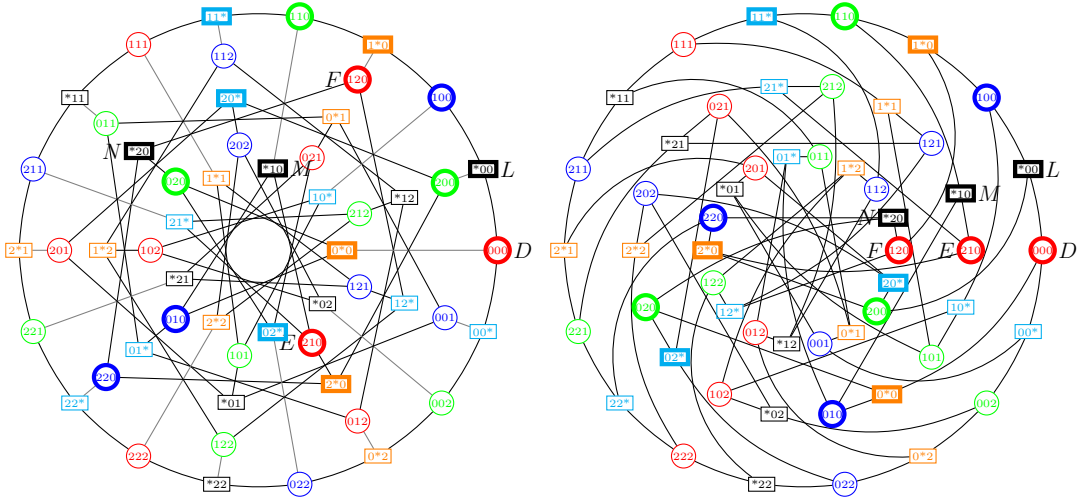
\begin{figure}[!h]
\begin{center}

\begin{tikzpicture}[vtx/.style = {draw, circle, inner sep = .5 pt, font = \tiny}, lin/.style = { draw, inner sep = 1.5 pt, font = \tiny}, scale = .7, every node/.style={transform shape}]

\pgfmathsetmacro{\r}{4.5}
\pgfmathsetmacro{\rr}{\r/10}
\pgfmathsetmacro{\er}{5.5}
\pgfmathsetmacro{\fr}{1.7}

\pgfmathsetmacro{\lr}{0}
\pgfmathsetmacro{\mr}{6.5}
\pgfmathsetmacro{\nr}{3.5}

\pgfmathsetmacro{\loff}{180/9}
\pgfmathsetmacro{\eoff}{-3*180/9}
\pgfmathsetmacro{\moff}{4*180/9}
\pgfmathsetmacro{\foff}{3*180/9}
\pgfmathsetmacro{\noff}{4*180/9+3*180/9}

\node[vtx,red, ultra thick] (D0) at (360*0/9: \r){000};
\node[vtx,blue, ultra thick] (D1) at (360*1/9: \r){100};
\node[vtx,green, ultra thick] (D2) at (360*2/9: \r){110};
\node[vtx,red] (D3) at (360*3/9: \r){111};
\node[vtx,blue] (D4) at (360*4/9: \r){211};
\node[vtx,green] (D5) at (360*5/9: \r){221};
\node[vtx,red] (D6) at (360*6/9: \r){222};
\node[vtx,blue] (D7) at (360*7/9: \r){022};
\node[vtx,green] (D8) at (360*8/9: \r){002};
\node[right  = -1mm of D0] {$D$};

\node[lin,black, ultra thick] (L0) at (360*0/9+\loff: \r-\lr*\rr){*00};
\node[lin,orange, ultra thick] (L1) at (360*1/9+\loff: \r-\lr*\rr){1*0};
\node[lin,cyan, ultra thick] (L2) at (360*2/9+\loff: \r-\lr*\rr){11*};
\node[lin,black] (L3) at (360*3/9+\loff: \r-\lr*\rr){*11};
\node[lin,orange] (L4) at (360*4/9+\loff: \r-\lr*\rr){2*1};
\node[lin,cyan] (L5) at (360*5/9+\loff: \r-\lr*\rr){22*};
\node[lin,black] (L6) at (360*6/9+\loff: \r-\lr*\rr){*22};
\node[lin,orange] (L7) at (360*7/9+\loff: \r-\lr*\rr){0*2};
\node[lin,cyan] (L8) at (360*8/9+\loff: \r-\lr*\rr){00*};
\node[right  = -1mm of L0] {$L$};

\node[vtx,red, ultra thick] (E0) at (360*0/9+\eoff: \r-\er*\rr){210};
\node[vtx,blue] (E1) at (360*1/9+\eoff: \r-\er*\rr){121};
\node[vtx,green] (E2) at (360*2/9+\eoff: \r-\er*\rr){212};
\node[vtx,red] (E3) at (360*3/9+\eoff: \r-\er*\rr){021};
\node[vtx,blue] (E4) at (360*4/9+\eoff: \r-\er*\rr){202};
\node[vtx,green, ultra thick] (E5) at (360*5/9+\eoff: \r-\er*\rr){020};
\node[vtx,red] (E6) at (360*6/9+\eoff: \r-\er*\rr){102};
\node[vtx,blue, ultra thick] (E7) at (360*7/9+\eoff: \r-\er*\rr){010};
\node[vtx,green] (E8) at (360*8/9+\eoff: \r-\er*\rr){101};
\node[left  = -1mm of E0] {$E$};

\node[lin,black,  ultra thick] (M0) at (360*0/9+\moff: \r-\mr*\rr){*10};
\node[lin,orange] (M1) at (360*1/9+\moff: \r-\mr*\rr){1*1};
\node[lin,cyan] (M2) at (360*2/9+\moff: \r-\mr*\rr){21*};
\node[lin,black] (M3) at (360*3/9+\moff: \r-\mr*\rr){*21};
\node[lin,orange] (M4) at (360*4/9+\moff: \r-\mr*\rr){2*2};
\node[lin,cyan, ultra thick] (M5) at (360*5/9+\moff: \r-\mr*\rr){02*};
\node[lin,black] (M6) at (360*6/9+\moff: \r-\mr*\rr){*02};
\node[lin, orange,ultra thick] (M7) at (360*7/9+\moff: \r-\mr*\rr){0*0};
\node[lin,cyan] (M8) at (360*8/9+\moff: \r-\mr*\rr){10*};
\node[right  = -1mm of M0] {$M$};

\node[vtx,red,,ultra thick] (F0) at (360*0/9+\foff: \r-\fr*\rr){120};
\node[vtx,blue] (F1) at (360*1/9+\foff: \r-\fr*\rr){112};
\node[vtx,green] (F2) at (360*2/9+\foff: \r-\fr*\rr){011};
\node[vtx,red] (F3) at (360*3/9+\foff: \r-\fr*\rr){201};
\node[vtx,blue, ultra thick] (F4) at (360*4/9+\foff: \r-\fr*\rr){220};
\node[vtx,green] (F5) at (360*5/9+\foff: \r-\fr*\rr){122};
\node[vtx,red] (F6) at (360*6/9+\foff: \r-\fr*\rr){012};
\node[vtx,blue] (F7) at (360*7/9+\foff: \r-\fr*\rr){001};
\node[vtx,green, ultra thick] (F8) at (360*8/9+\foff: \r-\fr*\rr){200};
\node[left  = -1mm of F0] {$F$};

\node[lin,black, ,ultra thick] (N0) at (360*0/9+\noff: \r-\nr*\rr){*20};
\node[lin, orange] (N1) at (360*1/9+\noff: \r-\nr*\rr){1*2};
\node[lin,cyan] (N2) at (360*2/9+\noff: \r-\nr*\rr){01*};
\node[lin,black] (N3) at (360*3/9+\noff: \r-\nr*\rr){*01};
\node[lin,orange, ultra thick] (N4) at (360*4/9+\noff:\r-\nr*\rr){2*0};
\node[lin,cyan] (N5) at (360*5/9+\noff: \r-\nr*\rr){12*};
\node[lin,black] (N6) at (360*6/9+\noff: \r-\nr*\rr){*12};
\node[lin,orange] (N7) at (360*7/9+\noff: \r-\nr*\rr){0*1};
\node[lin,cyan, ultra thick] (N8) at (360*8/9+\noff: \r-\nr*\rr){20*};
\node[left  = -1mm of N0] {$N$};

\begin{scope}[on background layer]
\foreach \i in {0,1, ..., 8}
{
\draw let \n1 = {int(mod(\i+0, 9))} in (L\i) to[bend left = 5] (D\n1);
\draw let \n1 = {int(mod(\i+1, 9))} in (L\i) to[bend left = -5] (D\n1);
\draw[gray] let \n1 = {int(mod(\i+8, 9))} in (L\i) to[bend left = 0] (F\n1);

\draw let \n1 = {int(mod(\i+0, 9))} in (M\i) to[bend left = 0] (E\n1);
\draw let \n1 = {int(mod(\i+7, 9))} in (M\i) to[bend left = 0] (E\n1);
\draw[gray] let \n1 = {int(mod(\i+2, 9))} in (M\i) to[bend left =0] (D\n1);

\draw let \n1 = {int(mod(\i+0, 9))} in (N\i) to[bend left = 0] (F\n1);
\draw let \n1 = {int(mod(\i+4, 9))} in (N\i) to[bend left = 0] (F\n1);
\draw[] let \n1 = {int(mod(\i+5, 9))} in (N\i) to[bend left = 0] (E\n1);

}
\end{scope}

\end{tikzpicture}
\begin{tikzpicture}[vtx/.style = {draw, circle, inner sep = .5 pt, font = \tiny}, lin/.style = { draw, inner sep = 1.5 pt, font = \tiny}, scale = .7, every node/.style={transform shape}]

\pgfmathsetmacro{\r}{4.5}
\pgfmathsetmacro{\rr}{\r/10}
\pgfmathsetmacro{\er}{3}
\pgfmathsetmacro{\fr}{6}

\pgfmathsetmacro{\lr}{0}
\pgfmathsetmacro{\mr}{3}
\pgfmathsetmacro{\nr}{6}

\pgfmathsetmacro{\loff}{180/9}
\pgfmathsetmacro{\eoff}{0}
\pgfmathsetmacro{\moff}{180/9}
\pgfmathsetmacro{\foff}{0}
\pgfmathsetmacro{\noff}{180/9}

\node[vtx,red, ultra thick] (D0) at (360*0/9: \r){000};
\node[vtx,blue, ultra thick] (D1) at (360*1/9: \r){100};
\node[vtx,green, ultra thick] (D2) at (360*2/9: \r){110};
\node[vtx,red] (D3) at (360*3/9: \r){111};
\node[vtx,blue] (D4) at (360*4/9: \r){211};
\node[vtx,green] (D5) at (360*5/9: \r){221};
\node[vtx,red] (D6) at (360*6/9: \r){222};
\node[vtx,blue] (D7) at (360*7/9: \r){022};
\node[vtx,green] (D8) at (360*8/9: \r){002};
\node[right  = -1mm of D0] {$D$};

\node[lin,black, ultra thick] (L0) at (360*0/9+\loff: \r-\lr*\rr){*00};
\node[lin,orange, ultra thick] (L1) at (360*1/9+\loff: \r-\lr*\rr){1*0};
\node[lin,cyan, ultra thick] (L2) at (360*2/9+\loff: \r-\lr*\rr){11*};
\node[lin,black] (L3) at (360*3/9+\loff: \r-\lr*\rr){*11};
\node[lin,orange] (L4) at (360*4/9+\loff: \r-\lr*\rr){2*1};
\node[lin,cyan] (L5) at (360*5/9+\loff: \r-\lr*\rr){22*};
\node[lin,black] (L6) at (360*6/9+\loff: \r-\lr*\rr){*22};
\node[lin,orange] (L7) at (360*7/9+\loff: \r-\lr*\rr){0*2};
\node[lin,cyan] (L8) at (360*8/9+\loff: \r-\lr*\rr){00*};
\node[right  = -1mm of L0] {$L$};

\node[vtx,red, ultra thick] (E0) at (360*0/9+\eoff: \r-\er*\rr){210};
\node[vtx,blue] (E1) at (360*1/9+\eoff: \r-\er*\rr){121};
\node[vtx,green] (E2) at (360*2/9+\eoff: \r-\er*\rr){212};
\node[vtx,red] (E3) at (360*3/9+\eoff: \r-\er*\rr){021};
\node[vtx,blue] (E4) at (360*4/9+\eoff: \r-\er*\rr){202};
\node[vtx,green, ultra thick] (E5) at (360*5/9+\eoff: \r-\er*\rr){020};
\node[vtx,red] (E6) at (360*6/9+\eoff: \r-\er*\rr){102};
\node[vtx,blue, ultra thick] (E7) at (360*7/9+\eoff: \r-\er*\rr){010};
\node[vtx,green] (E8) at (360*8/9+\eoff: \r-\er*\rr){101};
\node[left  = -1mm of E0] {$E$};

\node[lin,black,  ultra thick] (M0) at (360*0/9+\moff: \r-\mr*\rr){*10};
\node[lin,orange] (M1) at (360*1/9+\moff: \r-\mr*\rr){1*1};
\node[lin,cyan] (M2) at (360*2/9+\moff: \r-\mr*\rr){21*};
\node[lin,black] (M3) at (360*3/9+\moff: \r-\mr*\rr){*21};
\node[lin,orange] (M4) at (360*4/9+\moff: \r-\mr*\rr){2*2};
\node[lin,cyan, ultra thick] (M5) at (360*5/9+\moff: \r-\mr*\rr){02*};
\node[lin,black] (M6) at (360*6/9+\moff: \r-\mr*\rr){*02};
\node[lin, orange,ultra thick] (M7) at (360*7/9+\moff: \r-\mr*\rr){0*0};
\node[lin,cyan] (M8) at (360*8/9+\moff: \r-\mr*\rr){10*};
\node[right  = -1mm of M0] {$M$};

\node[vtx,red,,ultra thick] (F0) at (360*0/9+\foff: \r-\fr*\rr){120};
\node[vtx,blue] (F1) at (360*1/9+\foff: \r-\fr*\rr){112};
\node[vtx,green] (F2) at (360*2/9+\foff: \r-\fr*\rr){011};
\node[vtx,red] (F3) at (360*3/9+\foff: \r-\fr*\rr){201};
\node[vtx,blue, ultra thick] (F4) at (360*4/9+\foff: \r-\fr*\rr){220};
\node[vtx,green] (F5) at (360*5/9+\foff: \r-\fr*\rr){122};
\node[vtx,red] (F6) at (360*6/9+\foff: \r-\fr*\rr){012};
\node[vtx,blue] (F7) at (360*7/9+\foff: \r-\fr*\rr){001};
\node[vtx,green, ultra thick] (F8) at (360*8/9+\foff: \r-\fr*\rr){200};
\node[left  = -1mm of F0] {$F$};

\node[lin,black, ,ultra thick] (N0) at (360*0/9+\noff: \r-\nr*\rr){*20};
\node[lin, orange] (N1) at (360*1/9+\noff: \r-\nr*\rr){1*2};
\node[lin,cyan] (N2) at (360*2/9+\noff: \r-\nr*\rr){01*};
\node[lin,black] (N3) at (360*3/9+\noff: \r-\nr*\rr){*01};
\node[lin,orange, ultra thick] (N4) at (360*4/9+\noff:\r-\nr*\rr){2*0};
\node[lin,cyan] (N5) at (360*5/9+\noff: \r-\nr*\rr){12*};
\node[lin,black] (N6) at (360*6/9+\noff: \r-\nr*\rr){*12};
\node[lin,orange] (N7) at (360*7/9+\noff: \r-\nr*\rr){0*1};
\node[lin,cyan, ultra thick] (N8) at (360*8/9+\noff: \r-\nr*\rr){20*};
\node[left  = -1mm of N0] {$N$};

\begin{scope}[on background layer]
\foreach \i in {0,1, ..., 8}
{
\draw let \n1 = {int(mod(\i+0, 9))} in (L\i) to[bend left = 5] (D\n1);
\draw let \n1 = {int(mod(\i+1, 9))} in (L\i) to[bend left = -5] (D\n1);
\draw[] let \n1 = {int(mod(\i+8, 9))} in (L\i) to[bend left = 40] (F\n1);

\draw let \n1 = {int(mod(\i+0, 9))} in (M\i) to[bend left = 0] (E\n1);
\draw let \n1 = {int(mod(\i+7, 9))} in (M\i) to[bend left = 0] (E\n1);
\draw[] let \n1 = {int(mod(\i+2, 9))} in (M\i) to[bend left =-20] (D\n1);

\draw let \n1 = {int(mod(\i+0, 9))} in (N\i) to[bend left = 0] (F\n1);
\draw let \n1 = {int(mod(\i+4, 9))} in (N\i) to[bend left = 0] (F\n1);
\draw[] let \n1 = {int(mod(\i+5, 9))} in (N\i) to[bend left = -15] (E\n1);

}
\end{scope}

\end{tikzpicture}

\caption{Two drawings of the Gray Graph with $\mathbb{Z}_{9}$ symmetry. Note that the symmetry classes under the Pappus action (the columns in Table \ref{tab:SymmetryClassLabels}) are preserved under 3-fold rotation applied to this graph. 
The elements $*_{0j}$, $* \in \{X,Y,Z, R,B,G\}$ (that is, the first rows in Table \ref{tab:SymmetryClassLabels})  are shown thick.}
\label{fig:GrayZ9}
\end{center}
\end{figure}

Table~\ref{tab:Z9-symmetryClasses} lists the symmetry classes of points and lines that correspond to this realization of the Gray Graph using $\mathbb{Z}_{9}$ 
symmetry, shown in Figure \ref{fig:GrayZ9}. They are chosen so that 3-fold rotation permutes the ``Pappus'' symmetry classes (that is, the colored columns in 
Table~\ref{tab:SymmetryClassLabels}.

\begin{table}
\caption{Symmetry classes of points and lines corresponding to a $\mathbb{Z}_{9}$ realization of the Gray Graph.}
\label{tab:Z9-symmetryClasses}
\begin{align*}
\text{symmetry point class $D$: }& R_{00}, B_{00}, G_{00}, R_{10}, B_{10}, G_{10}, R_{20}, B_{20}, G_{20}   \\
=&\  000, 100, 110, 111, 211, 221, 222, 022, 002\\
\text{symmetry line class $L$: }& X_{00}, Y_{00}, Z_{00}, X_{10}, Y_{21}, Z_{10}, X_{20}, Y_{20}, Z_{20} \\  
=&\  \!*\!00, 1\!*\!0, 11*, *11, 2\!*\!1, 22*, *22, 0\!*\!2, 00*\\
\text{symmetry point class $F$: }&  R_{02}, B_{22}, G_{12},R_{12}, B_{02}, G_{22}, R_{22}, B_{12}, G_{02}\\
=&\ 120, 112, 011, 201, 220, 122, 012, 001, 200\\
\text{symmetry line class $N$: }& X_{02}, Y_{22}, Z_{12}, X_{12}, Y_{02}, Z_{22}, X_{22}, Y_{12}, Z_{02}\\ 
=&\ \!*\!20, 12*, 01*, *01, 2\!*\!0, 12*, *12, 0\!*\!1, 20*\\
\text{symmetry point class $E$: }& R_{01}, B_{11}, G_{21}, R_{11}, B_{21}, G_{01}, R_{21}, B_{01}, G_{11}\\
=& \ 210, 121,212,012,202,020,102,010,101\\
\text{symmetry line class $M$: }& X_{01}, Y_{11}, Z_{21}, X_{11}, Y_{21}, Z_{01}, X_{21}, Y_{01}, Z_{11}\\
= &\  \!*\!10, 1\!*\!1, 21*, *21,2\!*\!2,02*,*02,0\!*\!0, 10*
\end{align*}
\end{table}

The symmetry classes under the $GG$ action (the color-coded rows in Table \ref{tab:SymmetryClassLabels}) are 
preserved through interchanging the rings of symmetry classes, but the action is more complicated. Simply cyclically 
permuting the three rings (outside-middle-inside) in the second drawing in Figure  \ref{fig:GrayZ9} permutes the 
elements in the classes $R_{0i}, R_{1i}, R_{2i}$ and $X_{0i}, X_{1i}, X_{2i}$: for example, mapping the outer ring to the 
middle ring to the center ring applies the permutation $R_{00} = 000 \to R_{01} = 210 \to R_{02} = 120$. However, the 
permuting of the rings does not map the other $GG$ symmetry classes to themselves directly. To preserve the classes 
$B_{ij}$ (blue) and $Y_{ij}$ (orange), $i = 0,1,2$, permuting the rings plus a $-120^{\circ}$ rotation is required: for 
example, mapping the outer ring to the middle ring and rotating backwards by $120^{\circ}$ sends $B_{00} = 100 \to 
B_{01} = 010$, and doing that action again sends $B_{01} = 010 \to B_{02} = 220$. Similarly, to preserve $GG$ classes $G_{ij}
$ (green) and $Z_{ij}$ (cyan), $i = 0,1,2$, requires a ring permutation and a rotation by $+120^{\circ}$.

\subsection{Constructing a $\mathbb{Z}_{9}$ realization}

We use the 9-fold rotation of the graph shown in Figure \ref{fig:GrayZ9} to construct orbits of length 9 of points and lines, listed in 
Table~\ref{tab:Z9-symmetryClasses}. The 0th element of each orbit is shown thick in Figure \ref{fig:GrayZ9}. These symmetry classes 
can be seen as 9-cycles on the standard grid, viewed as a solid torus formed by identifying opposite sides of the $3 \times 3$ grid, 
shown in Figure~\ref{fig:Z9GrayGrid}. However, producing the symmetry classes listed in Table \ref{tab:Z9-symmetryClasses} is not 
as straightforward as just following the 9-gons. The sequence of points and lines obtained by following the solid 9-gon on the torus 
corresponds to alternating points in class $D$ and lines in class $L$. However, to alternate between $F$ and $N$ requires skipping 2 
steps on the doubled 9-gon, and to alternate between $E$ and $M$ requires skipping 4 steps on the dashed 9-gon.

\begin{figure}[!h]
\begin{center}

\def\lw{.4}

\begin{tikzpicture}  [
  3d view,perspective
,
 ]  
\foreach \i in {0,1,2}{
\foreach \j in {0,1,2}{
\foreach \k in {0,1,2}{
\node[draw, circle, font = \scriptsize, fill = white, inner sep = 1pt] (\i\j\k) at ($3*(\i,\j,\k)$) {$\i\j\k$};
}}}

\path (000) node [draw, circle, ultra thick,inner sep = 0pt]{$\phantom{000}$};
\path(120) node[draw, circle, ultra thick,inner sep = 0pt] {$\phantom{120}$};
\path(210) node[draw, circle, ultra thick,inner sep = 0pt] {$\phantom{210}$};

\begin{scope}[on background layer]
\foreach \i in {0,1,2}{
\foreach \j in {0,1,2}{
\draw [cyan,shorten >=-.75cm, line width = \lw mm] (\i\j0) -- (\i\j2);
\draw [orange,shorten >=-.75cm, line width = \lw mm] (\i0\j) -- (\i2\j);
\draw [black,shorten >=-.85 cm, line width = .1 mm] (0\i\j) -- (2\i\j);
}}


\foreach \i/\j in {0/2,2/2,0/1,1/1,1/0,2/0}{
\path ($(\i\j0)!1.11!(\i\j2)$) node[cyan, font=\scriptsize] {${\i\j*}$};
}
\path ($(120)!1.23!(122)$) node[cyan, font=\scriptsize] {${12*}$};
\path ($(210)!1.23!(212)$) node[cyan, font=\scriptsize] {${21*}$};
\path ($(000)!1.1!(002)$) node[cyan, font=\scriptsize, right] {${00*}$};

\foreach \i/\j in {1/2,2/2,0/1,2/1,0/0,1/0}{
\path ($(\i0\j)!1.23!(\i2\j)$) node[orange, font=\scriptsize] {${\i\!*\!\j}$};
}

\path ($(002)!1.28!(022)$) node[orange, font=\scriptsize] {${0\!*\!2}$};
\path ($(101)!1.2!(121)$) node[orange, font=\scriptsize, above] {${1\!*\!1}$};
\path ($(200)!1.2!(220)$) node[orange, font=\scriptsize,above] {${2\!*\!0}$};
\foreach \i/\j in {1/2,0/2,1/1,2/1,0/0,2/0}{
\path ($(0\i\j)!1.17!(2\i\j)$) node[black, font=\scriptsize,] {${*\i\j}$};
}

\path ($(022)!1.2!(222)$) node[black, font=\scriptsize, above] {${*22}$};
\path ($(001)!1.27!(201)$) node[black, font=\scriptsize] {${*01}$};
\path ($(010)!1.2!(210)$) node[black, font=\scriptsize,above] {${*10}$};

\draw[ultra thick, -latex, ] (000) -- (100) -- (110) -- (111) -- (211) -- (221) -- (222) -- ($(222)!-.2!(022)$);
\draw[ultra thick, -latex]($(222)!1.2!(022)$) -- (022);
\draw[ultra thick, latex-]($(022)!-.2!(002)$) -- (022) ;
\draw[ultra thick,, -latex]($(022)!1.2!(002)$) -- (002) -- ($(002)!-.13!(000)$) ;
\draw[ultra thick, -latex] ($(002)!1.2!(000)$)--(000) ;

\draw[ultra thick, dashed, latex-]($(210)!-.2!(010)$)-- (210);
\draw[ultra thick, dashed, -latex]($(210)!1.2!(010)$) -- (010);
\draw [ultra thick, dashed, -latex] (010)-- (020) -- (021)--(121); 
\draw[ultra thick, dashed, latex-]($(121)!-.25!(101)$)-- (121);
\draw[ultra thick, dashed, -latex]($(121)!1.25!(101)$) -- (101);
\draw[ultra thick, dashed, -latex](101) -- (102) -- (202) -- (212) ;
\draw[ultra thick, dashed, latex-]($(212)!-.2!(210)$) -- (212);
\draw[ultra thick, dashed, -latex]($(212)!1.2!(210)$) -- (210);

\draw[thick,double, latex-] (120) -- (220);

\draw[thick, ,double,-latex]($(220)!-.2!(200)$) -- (220);
\draw[thick,double,latex-]($(220)!1.2!(200)$) -- (200);

\draw[thick,double,latex-](200)-- (201) ;
\draw[thick,double,-latex]($(201)!-.2!(001)$) -- (201);
\draw[thick,double,latex-]($(201)!1.2!(001)$) -- (001);

\draw[thick,double,latex-]  (001) -- (011) -- (012) -- (112)--(122); 
\draw[thick,double,-latex]($(122)!-.2!(120)$) -- (122);
\draw[thick,double,latex-]($(122)!1.2!(120)$) -- (120);

\end{scope}
\end{tikzpicture}

\caption{Identifying $\mathbb{Z}_{9}$ symmetry classes in the Gray Grid. Successive points in symmetry class $D$ occur by traveling one step 
along the solid black 9-gon, successive points in symmetry class $F$ occur by traveling two steps along the doubled 9-gon, and successive points 
in symmetry class $E$ occur by travelling four steps along the dashed 9-gon (in each case, following the direction of the arrows). The 0th element 
of each point class is indicated by a heavy circle.}
\label{fig:Z9GrayGrid}
\end{center}
\end{figure}
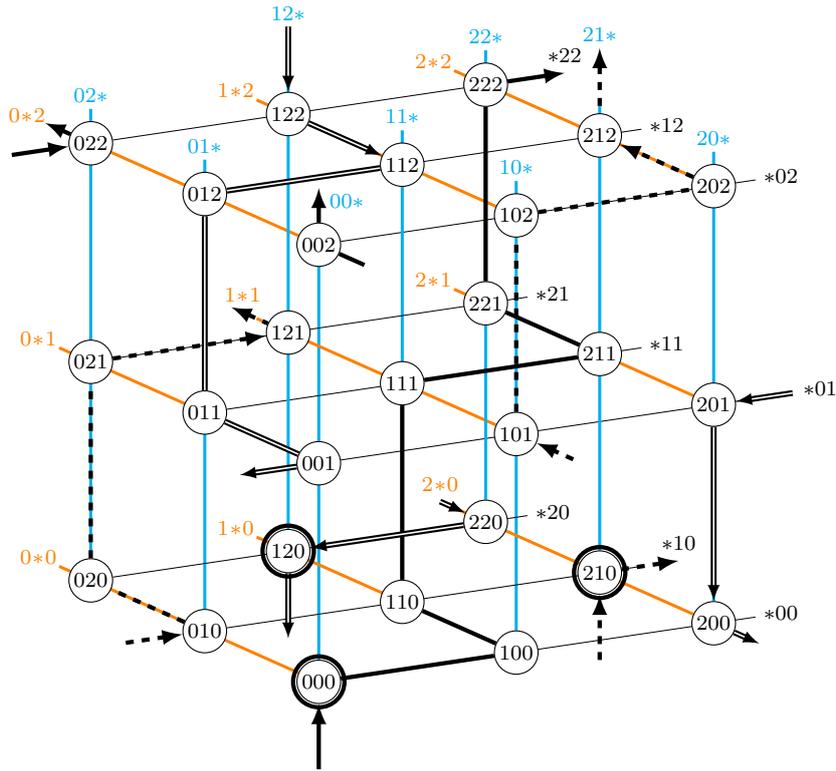

Using the drawings in Figure \ref{fig:GrayZ9}, it is straightforward to read off the voltages for the $\mathbb{Z}_{9}$ voltage graph, shown in 
Figure~\ref{fig:GrayZ9voltOriginal}. 
As usual, we then add and subtract voltages to produce a reduced Levi graph with a spanning path~\ref{fig:GrayZ9voltNice3}, to aid in applying 
known algorithms for constructing configurations with reduced Levi graphs of this type.

Figure~\ref{fig:grayZ9-3} shows the corresponding drawing of the Gray Graph emphasizing the $\mathbb{Z}_{9}$ symmetry classes.

\begin{figure}
\begin{center}
\begin{subfigure}[t]{.4\linewidth}
\centering
\begin{tikzpicture}[vtx/.style = {draw, circle, inner sep = .5 pt, }, line/.style = { draw, inner sep = 1.5 pt, },lbl/.style={midway, font = \scriptsize, inner sep = 2, fill = white}]

\pgfmathsetmacro{\r}{1.5}
\node[vtx, orange] (D) at (360*0/6:\r){$D$};
\node[line, orange] (L) at (360*1/6:\r){$L$};
\node[vtx, magenta] (F) at (360*2/6:\r){$F$};
\node[line, magenta] (N) at (360*3/6:\r){$N$};
\node[vtx,cyan] (E) at (360*4/6:\r){$E$};
\node[line,cyan ] (M) at (360*5/6:\r){$M$};

\draw[->] (L) to[bend left = 30] node[lbl]{1} (D);
\draw[->, ultra thick] (L) to[bend right = 30] (D);
\draw[->,] (L) to[] node[lbl]{-1}  (F);

\draw[->] (M) to[bend left = 30] node[lbl]{-2} (E);
\draw[->, ultra thick] (M) to[bend right = 30] (E);
\draw[->,] (M) to[] node[lbl]{2} (D);

\draw[->] (N) to[bend left = 30] node[lbl]{4} (F);
\draw[->, ultra thick] (N) to[bend right = 30] (F);
\draw[->, ] (N) to[]  node[lbl]{5} (E);

\node[below = of D] {$\mathbb{Z}_{9}$};

\end{tikzpicture}
\caption{The reduced Levi graph corresponding to the graphs shown in Figure \ref{fig:GrayZ9}. Thick edges correspond to labels of 0.}
\label{fig:GrayZ9voltOriginal}
\end{subfigure}
\hspace{1cm}
\begin{subfigure}[t]{.3\linewidth}
\centering
\begin{tikzpicture}[vtx/.style = {draw, circle, inner sep = .5 pt, }, line/.style = { draw, inner sep = 1.5 pt, },lbl/.style={midway, font = \scriptsize, inner sep = 2, fill = white}]

\pgfmathsetmacro{\r}{1.5}
\node[vtx, orange] (D) at (360*0/6:\r){$D$};
\node[line, orange] (L) at (360*1/6:\r){$L$};
\node[vtx, magenta] (F) at (360*2/6:\r){$F$};
\node[line,magenta] (N) at (360*3/6:\r){$N$};
\node[vtx,cyan] (E) at (360*4/6:\r){$E$};
\node[line, cyan] (M) at (360*5/6:\r){$M$};

\draw[->] (L) to[bend left = 30] node[lbl]{1} (D);
\draw[->, ultra thick] (L) to[bend right = 30] (D);
\draw[->,ultra thick] (L) to[]   (F);

\draw[->] (N) to[bend left = 30] node[lbl]{4} (F);
\draw[->, ultra thick] (N) to[bend right = 30] (F);
\draw[->, ultra thick] (N) to[]   (E);

\draw[->] (M) to[bend left = 30] node[lbl]{2} (E);
\draw[->, ultra thick] (M) to[bend right = 30] (E);
\draw[->] (M) to[] node[lbl]{-1} (D);

\node[below = of D] {$\mathbb{Z}_{9}$};

\end{tikzpicture}
\caption{Adding and subtracting voltages as necessary results in a $\mathbb{Z}_{9}$ voltage graph with a Hamiltonian spanning path of 0-labeled edges.}
\label{fig:GrayZ9voltNice3}
\end{subfigure}
\hfill
\begin{subfigure}[t]{.68\linewidth}
\centering
\begin{tikzpicture}[vtx/.style = {draw, circle, inner sep = .5 pt, font = \tiny}, lin/.style = { draw, inner sep = 1.5 pt, font = \tiny}]

\pgfmathsetmacro{\r}{5.2}
\pgfmathsetmacro{\rr}{\r/8}

\pgfmathsetmacro{\lr}{1}
\pgfmathsetmacro{\fr}{3}
\pgfmathsetmacro{\nr}{4}
\pgfmathsetmacro{\er}{6}
\pgfmathsetmacro{\mr}{7}

\pgfmathsetmacro{\loff}{0}
\pgfmathsetmacro{\eoff}{0}
\pgfmathsetmacro{\moff}{0}
\pgfmathsetmacro{\foff}{0}
\pgfmathsetmacro{\noff}{0}

\node[vtx,
orange, ultra thick] (D0) at (360*0/9: \r){000};
\node[vtx,
orange] (D1) at (360*1/9: \r){100};
\node[vtx,
,orange] (D2) at (360*2/9: \r){110};
\node[vtx
,orange] (D3) at (360*3/9: \r){111};
\node[vtx,
,orange](D4) at (360*4/9: \r){211};
\node[vtx,
,orange](D5) at (360*5/9: \r){221};
\node[vtx,
,orange] (D6) at (360*6/9: \r){222};
\node[vtx,
,orange] (D7) at (360*7/9: \r){022};
\node[vtx,
,orange](D8) at (360*8/9: \r){002};
\node[below  = -1mm of D0] {$D$};

\node[lin,
,orange, ultra thick] (L0) at (360*0/9+\loff: \r-\lr*\rr){*00};
\node[lin,orange] (L1) at (360*1/9+\loff: \r-\lr*\rr){1*0};
\node[lin,
,orange](L2) at (360*2/9+\loff: \r-\lr*\rr){11*};
\node[lin,
,orange] (L3) at (360*3/9+\loff: \r-\lr*\rr){*11};
\node[lin,orange] (L4) at (360*4/9+\loff: \r-\lr*\rr){2*1};
\node[lin,
,orange] (L5) at (360*5/9+\loff: \r-\lr*\rr){22*};
\node[lin,
,orange] (L6) at (360*6/9+\loff: \r-\lr*\rr){*22};
\node[lin,orange] (L7) at (360*7/9+\loff: \r-\lr*\rr){0*2};
\node[lin,
,orange] (L8) at (360*8/9+\loff: \r-\lr*\rr){00*};
\node[below  = -1mm of L0] {$L$};

\node[vtx,
magenta,ultra thick] (F0) at (360*0/9+\foff: \r-\fr*\rr){200};
\node[vtx,
magenta,] (F1) at (360*1/9+\foff: \r-\fr*\rr){120};
\node[vtx,magenta] (F2) at (360*2/9+\foff: \r-\fr*\rr){112};
\node[vtx,magenta] (F3) at (360*3/9+\foff: \r-\fr*\rr){011};
\node[vtx,magenta] (F4) at (360*4/9+\foff: \r-\fr*\rr){201};
\node[vtx,magenta, ] (F5) at (360*5/9+\foff: \r-\fr*\rr){220};
\node[vtx,magenta] (F6) at (360*6/9+\foff: \r-\fr*\rr){122};
\node[vtx,magenta] (F7) at (360*7/9+\foff: \r-\fr*\rr){012};
\node[vtx,magenta] (F8) at (360*8/9+\foff: \r-\fr*\rr){001};
\node[below  = -1mm of F0] {$F$};

\node[lin,magenta ,ultra thick] (N0) at (360*0/9+\noff:\r-\nr*\rr){2*0};
\node[lin,magenta] (N1) at (360*1/9+\noff: \r-\nr*\rr){12*};
\node[lin,magenta] (N2) at (360*2/9+\noff: \r-\nr*\rr){*12};
\node[lin,magenta] (N3) at (360*3/9+\noff: \r-\nr*\rr){0*1};
\node[lin,magenta] (N4) at (360*4/9+\noff: \r-\nr*\rr){20*};
\node[lin,magenta,] (N5) at (360*5/9+\noff: \r-\nr*\rr){*20};
\node[lin, magenta] (N6) at (360*6/9+\noff: \r-\nr*\rr){1*2};
\node[lin,magenta] (N7) at (360*7/9+\noff: \r-\nr*\rr){01*};
\node[lin,magenta] (N8) at (360*8/9+\noff: \r-\nr*\rr){*01};
\node[below  = -1mm of N0] {$N$};

\node[vtx,cyan, ultra thick] (E0) at (360*0/9+\eoff: \r-\er*\rr){202};
\node[vtx,cyan] (E1) at (360*1/9+\eoff: \r-\er*\rr){020};
\node[vtx,cyan] (E2) at (360*2/9+\eoff: \r-\er*\rr){102};
\node[vtx,cyan] (E3) at (360*3/9+\eoff: \r-\er*\rr){010};
\node[vtx,cyan] (E4) at (360*4/9+\eoff: \r-\er*\rr){101};
\node[vtx,cyan, ] (E5) at (360*5/9+\eoff: \r-\er*\rr){210};
\node[vtx,cyan] (E6) at (360*6/9+\eoff: \r-\er*\rr){121};
\node[vtx,cyan] (E7) at (360*7/9+\eoff: \r-\er*\rr){212};
\node[vtx,cyan] (E8) at (360*8/9+\eoff: \r-\er*\rr){021};
\node[below  = -1mm of E0] {$E$};

\node[lin,cyan, ultra thick] (M0) at (360*0/9+\moff: \r-\mr*\rr){*02};
\node[lin, cyan,] (M1) at (360*1/9+\moff: \r-\mr*\rr){0*0};
\node[lin,cyan] (M2) at (360*2/9+\moff: \r-\mr*\rr){10*};
\node[lin,cyan, ] (M3) at (360*3/9+\moff: \r-\mr*\rr){*10};
\node[lin,cyan] (M4) at (360*4/9+\moff: \r-\mr*\rr){1*1};
\node[lin,cyan] (M5) at (360*5/9+\moff: \r-\mr*\rr){21*};
\node[lin,cyan] (M6) at (360*6/9+\moff: \r-\mr*\rr){*21};
\node[lin,cyan] (M7) at (360*7/9+\moff: \r-\mr*\rr){2*2};
\node[lin,cyan] (M8) at (360*8/9+\moff: \r-\mr*\rr){02*};
\node[left  = -1mm of M0] {$M$};

\begin{scope}[on background layer]
\foreach \i in {0,1, ..., 8}
{
\draw let \n1 = {int(mod(\i+0, 9))} in (L\i) to[bend left = 0] (D\n1);
\draw let \n1 = {int(mod(\i+1, 9))} in (L\i) to[bend left = -20] (D\n1);
\draw[] let \n1 = {int(mod(\i+0, 9))} in (L\i) to[bend left = 0] (F\n1);

\draw let \n1 = {int(mod(\i+0, 9))} in (N\i) to[bend left = 0] (F\n1);
\draw[] let \n1 = {int(mod(\i+4, 9))} in (N\i) to[bend left = -40] (F\n1);
\draw[] let \n1 = {int(mod(\i+0, 9))} in (N\i) to[bend left = -0] (E\n1);

\draw[] let \n1 = {int(mod(\i+0, 9))} in (M\i) to[bend left = 0] (E\n1);
\draw[] let \n1 = {int(mod(\i+2, 9))} in (M\i) to[bend left = -35] (E\n1);
\draw[] let \n1 = {int(mod(\i+8, 9))} in (M\i) to[bend left =28] (D\n1);
}

\foreach \i in {0}
{
\draw[ultra thick] let \n1 = {int(mod(\i+0, 9))} in (L\i) to[bend left = 0] (D\n1);
\draw[ultra thick] let \n1 = {int(mod(\i+0, 9))} in (L\i) to[bend left = 0] (F\n1);

\draw[ultra thick] let \n1 = {int(mod(\i+0, 9))} in (N\i) to[bend left = 0] (F\n1);
\draw[ultra thick] let \n1 = {int(mod(\i+0, 9))} in (N\i) to[bend left = -0] (E\n1);

\draw[ultra thick] let \n1 = {int(mod(\i+0, 9))} in (M\i) to[bend left = 0] (E\n1);
}
\end{scope}

\end{tikzpicture}
\caption{The Gray graph viewed as the expanded Levi graph, emphasizing the spanning path of 0 indicated in the voltage graph (which can be seen following the ``spine'' of the graph, shown thick). As usual, points are circular nodes and lines are rectangular nodes.
}
\label{fig:grayZ9-3}
\end{subfigure}

\caption{Helpful realizations of the Gray graph and the $\mathbb{Z}_{9}$ voltage graph.}
\label{fig:Z9Good}
\end{center}
\end{figure}

The voltage graph shown in Figure \ref{fig:GrayZ9voltNice3} is an example of a voltage graph that corresponds to a \emph{multilateral chiral 3-configuration}, as described in \cite{Ber2013} as a configuration whose reduced Levi graph over some $\mathbb{Z}_{m}$ is 3-regular and alternates double (parallel) arcs and single arcs. That paper provided an algorithm for constructing corresponding geometric configurations. The Configuration Construction Lemma, described in that paper and elsewhere, says, essentially, the following: Given a set of points $P_{i}$ that are cyclically labeled as the vertices of a regular $m$-gon centered at the origin $\mc{O}$, construct the circle $C$ passing through points $P_{b}$, $\mc{O}$, $P_{b-d}$. If a point $Q$ lies on $C$, and if $Q'$ is the rotation of $Q$ by $\frac{2 \pi d}{m}$, then the line $QQ'$ passes through $P_{b}$. This is particularly useful if the point $Q$ is constructed as the intersection of some other line constructed in the configuration with the circle $C$. In this case, to realize the given reduced Levi graph, the required process is as follows:

\begin{algorithm}\label{alg:Z9construction} To construct a geometric realization of a 3-configuration with the reduced Levi graph in Figure \ref{fig:GrayZ9voltNice3}, do the following (with index arithmetic modulo 9):
\begin{enumerate} 
\item Construct points $D_{i}$, $i = 0, 1, \ldots, 8$ as the vertices of a regular 9-gon; specifically, let $D_{i} = (\cos(2 \pi i/9), \sin(2 \pi i/9))$.
\item Construct lines $L_{i} = D_{i}D_{i+1}$;
\item Place a point $F_{0}(t)$ arbitrarily (parameterized by $t$) on line $L_{0}$ and let $F_{i}(t)$ be the rotation of $F_{0}$ through $2 \pi i/9$ about the origin;
\item Construct lines $N_{i}(t) = F_{i}(t)F_{i+4}(t)$
\item Construct a circle $\mc{C}$ through the three points $D_{-1}$, $D_{-1-2}$ and the origin;
\item Construct point $E_{0}$ to be the intersection of line $N_{0}(t)$ with $\mc{C}$, if it exists. If no point of intersection exists, then {\bf the algorithm fails}. 
If a point of intersection exists, let $E_{i}$ be the rotation of $E_{0}$ through $2 \pi i/9$ about the origin;
\item Construct lines $M_{i} = E_{i}E_{i+2}$. The line $M_{i}$ will pass through the point $D_{i-1}$.
\end{enumerate}
\end{algorithm}

\begin{theorem}\label{thm:nonrealizeZ9}
There are exactly two positions of $F_{0}(t)$ on $L_{0}$ so that the line $N_{0}(t) = F_{0}(t)F_{4}(t)$ intersects the circle $\mc{C}$ passing through the three points 
$D_{-1}$, $D_{-3}$ and $\mc{O}$. These two positions are precisely the points  $P_{1} =  L_{0} \cap L_{-3} $ and  $P_{2} = L_{0} \cap L_{3}$, and at these positions, 
the line $N_{0}(t)$ is tangent to $\mc{C}$. For all other values of $t$, the line $N_{0}(t)$ does not intersect $\mc{C}$. Therefore, any straight-line realization of the 
Gray Configuration with $\mathbb{Z}_{9}$ symmetry is a weak realization, because there are extra incidences caused by the fact that $F_{0}$ lies on two lines $L_{i}$ 
(rather than only one).
\end{theorem}

To prove this theorem, we will use the following lemma:

\begin{lemma}\label{lem:parab}
Let $\ell$ be a line and let $\ell'$ be the rotate of $\ell$ through some angle $\theta$ about a point $\mc{O}$. Let $P$ be an arbitrary point on $\ell$ and $P'$ 
the rotate through $\theta$ about $\mc{O}$ of $P$ (thus $P'$ lies on $\ell'$). The envelope of the lines $\overline{PP'}$ is a parabola with focus $\mc{O}$ 
and directrix formed by the line $\overline{O_{\ell} O_{\ell'}}$, where $O_{\ell}$, $O_{\ell'}$ are formed by reflecting $\mc{O}$ over $\ell$, $\ell'$ respectively.
\end{lemma}

\begin{proof}
Recall that to construct a parabola with a given focus and tangent to two given lines, the directrix is formed by reflecting the focus over each of those two lines 
and joining the image points. Let $\mc{P}$ be the parabola with focus $\mc{O}$ that is tangent to the lines $\ell$ and $\ell'$. Using similar triangles and angle-chasing, 
it is straightforward to show that in the above situation, if we reflect $\mc{O}$ over the line $PP'$ to form the point $O_{PP'}$,  the three points $O_{\ell}$, $O_{\ell'}$ 
and $O_{PP'}$ are collinear; thus, the variable line $PP'$ is tangent to the parabola $\mc{P}$ for all choices of point $P$. In addition, if $M_{\ell}$, $M_{\ell'}$ are the 
feet of the perpendiculars to $\ell$, $\ell'$ passing through $\mc{O}$, the vertex $V$ of $\mc{P}$ is the midpoint of the segment $M_{\ell}M_{\ell'}$, and the line  
$M_{\ell}M_{\ell'}$ is tangent to $\mc{P}$ at $V$.
\end{proof}

\begin{proof}[Proof of Theorem \ref{thm:nonrealizeZ9}]
Consider the setup of Algorithm \ref{alg:Z9construction}, except for convenience, choose starting coordinates 
\[D_{i} 
 = \left(\cos\left( \frac{2 \pi i}{9} + \frac{17\pi}{18}\right), \sin\left( \frac{2 \pi i}{9} + \frac{17\pi}{18}\right) \right).\] 
 
With this choice of coordinates, by Lemma \ref{lem:parab}, the lines $L_{0}$ and $L_{4}$ (which is the rotate of line $L_{0}$ through the angle 
$\theta = 4\cdot\frac{ 2 \pi}{9}$ about $\mc{O}$) are tangent to a parabola $\mc{P}$ with focus at $\mc{O} = (0,0)$ and axis of symmetry on the $y$-axis. 
Using basic trigonometry, it is straightforward to show that the directrix of $\mc{P}$ is parallel to the $x$-axis and has equation 
$y = -2 \cos\left(\frac{2 \pi}{9}\right)\cos\left(\frac{4 \pi}{9}\right)$. See Figure \ref{fig:IllustratingTangency} for labels and details.

\begin{figure}[htbp]
\begin{center}
\includegraphics[width = .8\linewidth]{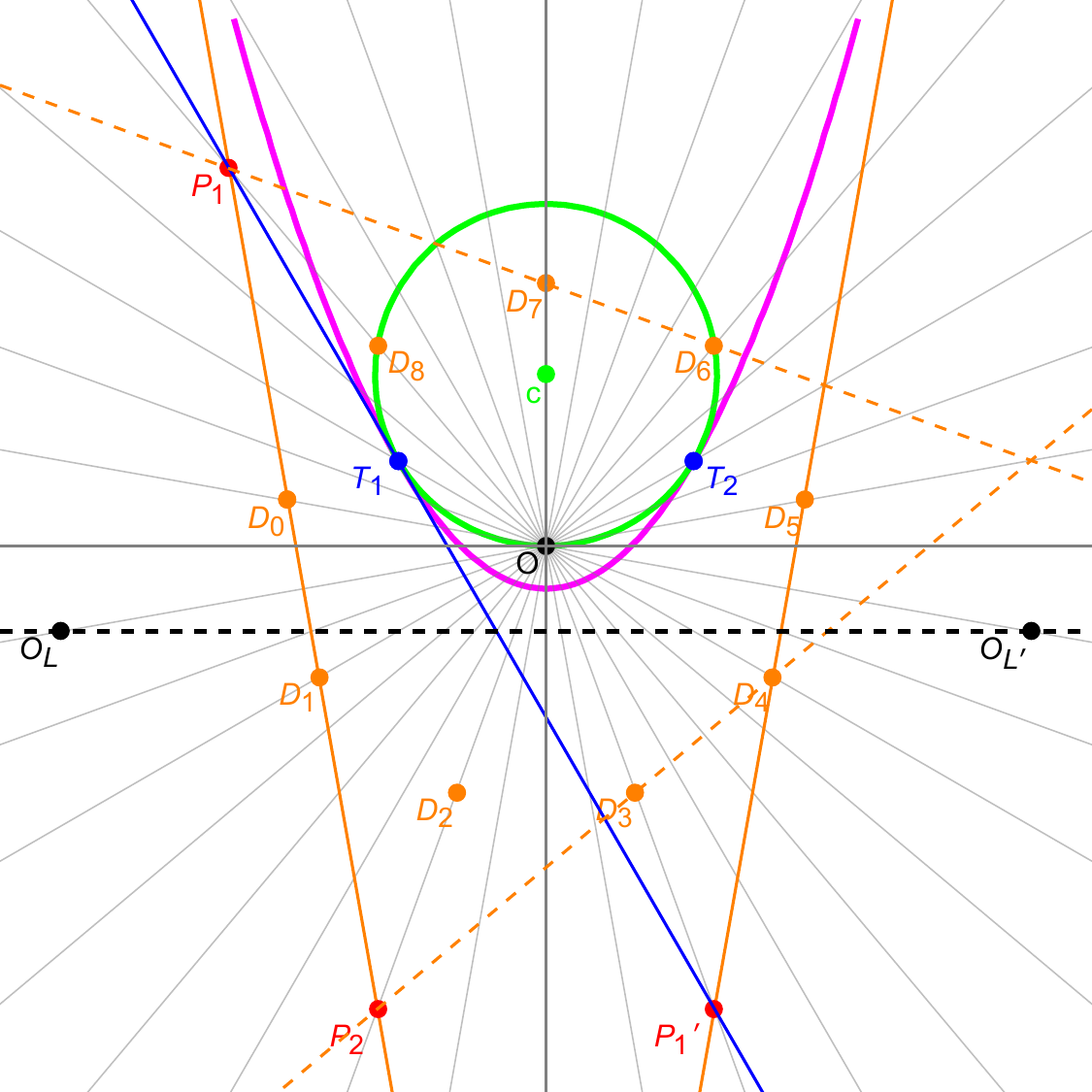}
\caption{Illustrating the proof of Theorem \ref{thm:nonrealizeZ9}. The circle $\mc{C}$ and its center are shown in green; the parabola 
$\mc{P}$ is shown in magenta with its directrix dashed black. The points $D_{i}$ and the lines $L_{0}$, $L_{4}$ are shown in orange, with 
$L_{-3}$ and $L_{3}$ shown in dashed orange. The points of tangency $T_{1}$ and $T_{2}$ between $\mc{P}$ and $\mc{C}$ are shown 
in blue, along with the line connecting $P_{1}$ and $P_{1}'$ (shown in red) that passes through $T_{1}$. The thin gray lines mark out 
angles of $\frac{\pi}{18}$.}
\label{fig:IllustratingTangency}
\end{center}
\end{figure}

Since the general formula of a parabola that opens up, whose axis of symmetry is the $y$-axis, and whose focus at the origin is $ x^2 = 4 p^2 + 4 p y$, 
where $p$ is half the distance from the focus to the directrix, it follows that the parabola $\mc{P}$ has equation
\begin{align*} x^{2} &= 4y\cos\left(\frac{2 \pi}{9}\right)\cos\left(\frac{4 \pi}{9}\right) + 4
\cos^{2}\left(\frac{2 \pi}{9}\right)\cos^{2}\left(\frac{4 \pi}{9}\right)\\
&= 4 y \sin \left(\frac{\pi }{18}\right) \cos
   \left(\frac{\pi }{9}\right)+4 \sin ^2\left(\frac{\pi
   }{18}\right) \cos ^2\left(\frac{\pi }{9}\right)
 \end{align*}
 after simplification.

\newpage

It is straightforward to show that the circle $\mc{C}$ passing through $D_{-1}, D_{-3}, \mc{O}$ has center  $c = (0, \frac{1}{2 \cos(2 \pi/9)} )$ and equation
\[
x^2+y^2-y \sec \left(\frac{2 \pi }{9}\right)=0.
\]

Define $F_{0}(t) = (1-t) D_{0} + t D_{1}$ to be a variable point on line $L_{0}$, and define its rotate $F_{4}(t) = (1-t) D_{4} + t D_{5}$; that is, $F_{4}(t)$ is 
the rotate of $F_{0}(t)$ through $\frac{4 \cdot 2 \pi}{9}$ about $\mc{O}$. As in Algorithm \ref{alg:Z9construction}, define $N_{0}(t) = F_{0}(t)F_{4}(t)$. 
By Lemma~\ref{lem:parab}, this (generic) line $N_{0}(t)$ is tangent to $\mc{P}$. Thus, to investigate which lines $N_{0}(t)$ intersect $\mc{C}$, we 
can first investigate the intersections of the parabola and the circle.
 
Using \emph{Mathematica}, we solve for the intersections of $\mc{C}$ and $\mc{P}$. No solutions would indicate that the circle and the parabola 
have no real intersections; four distinct solutions would indicate that the circle and the parabola intersect transversally; and two distinct solutions 
would show that the circle and the parabola intersect only at two tangent points. It is this third possiblity that turns out to be the case: after simplification, 
the only points of intersection between $\mc{P}$ and $\mc{C}$ are  two points of tangency
\begin{align*}
T_{1} &= \left( -2 \sqrt{3} \sin \left(\frac{\pi }{18}\right) \cos
   \left(\frac{\pi }{9}\right) , 2 \sin \left(\frac{\pi
   }{18}\right) \cos \left(\frac{\pi }{9}\right) \right) \\
   & \text{ and } \\ 
T_{2} &= \left( 2 \sqrt{3} \sin \left(\frac{\pi }{18}\right) \cos
   \left(\frac{\pi }{9}\right) , 2 \sin \left(\frac{\pi
   }{18}\right) \cos \left(\frac{\pi }{9}\right) \right).
\end{align*}

We claim that the tangent line to $T_{1}$ is precisely the line that passes through $P_{1} = L_{-3}\cap L_{0}$ and its rotate $P_{1}'$ through 
$\frac{4\cdot 2 \pi}{9}$ about $\mc{O}$. Elementary right-triangle trigonometry shows that the coordinates of $P_{1}$ and $P_{1}'$ are 
\begin{align*}
 P_{1} &= 2 \cos\left(\frac{\pi}{9}\right) \left(-\sin\left(\frac{2 \pi}{9}\right), \cos\left( \frac{2 \pi}{9}\right) \right) \\
 P_{1}'&= 2 \cos\left(\frac{\pi}{9}\right)\left(\sin \left(\frac{\pi}{9}\right), -\cos\left(\frac{\pi}{9}\right)\right)
 \end{align*}

Computing $\det \bpm P_{1} & P_{1}' & T_{1} \\
1 & 1 & 1 \epm$ and verifying, by \emph{Mathematica}, that the determinant equals 0 shows that $P_{1}$, $P_{1}'$ and $T_{1}$ are collinear. 
Since (by construction of $P_{1}'$) the line $\overline{P_{1}P_{1}'}$ is a member of the envelope of lines to the parabola,  it follows that 
$\overline{P_{1}P_{1}'}$ is tangent to $\mc{P}$ at the point $T_{1}$. 

Symmetry of the construction shows that if we use $P_{2} = L_{0}\cap L_{3}$ as the starting point on line $L_{0}$, that the corresponding line 
$P_{2}P_{2}'$ (with $P_{2}'$ the rotate of $P_{2}$ through $\frac{4\cdot 2 \pi}{9}$ about $\mc{O}$) is tangent to $\mc{P}$ at $T_{2}$.

In summary, there are precisely two points (namely, $P_{1}$ and $P_{2}$) on $L_{0}$ that can serve as the points $F_{0}(t)$ that have the property 
that the line $N_{0}(t)$ intersects $\mc{C}$ at some point $E_{0}$, following the labeling from Algorithm \ref{alg:Z9construction}, namely 
$E_{0} = T_{1}$ if $F_{0}(t) = P_{1}$, or $E_{0} = T_{2}$ if $F_{0}(t) = P_{2}$. 

However, each of these possible points $F_{0}(t)$ has the property that in addition to the line $L_{0}$ passing through them, another line $L_{i}$ 
also passes through them. 
\end{proof}

\begin{remark}
In fact, choosing either of these points as $F_{0}$ and completing Algorithm \ref{alg:Z9construction} (see Figure \ref{fig:Z9celestial}, which uses the 
points $D_{i}$ rotated back so that $D_{0} = (1,0)$) results in four points lying on each line and four lines passing through each point; the resulting 
incidence structure is actually the $(27_{4})$ celestial configuration $9\#(1,3;4,3;2,3)$) (see, e.g., \cite[Section 3.7]{Gru2009b} for details on 3-celestial 
configurations, where they are called $3$-astral configurations). These extra incidences mean that the construction from Algorithm~\ref{alg:Z9construction} 
produces only a weak realization of the Gray configuration.
\end{remark}

A pseudoline realization of a configuration is a drawing of a configuration in which lines are allowed to ``wiggle'' but any two can intersect at most once. 
(In the projective plane, any two pseudolines intersect exactly once.) See, for example, \cite{Gru2009b, Ber2008, HandDiscCompGeom2018}. 
Two pseudoline drawings (topological realizations) of the Gray Configuration with $\mathbb{Z}_{9}$ symmetry are shown in Figure~\ref{fig:Z9realization}. 
The left-hand figure indicates what the weak realization would be with straight lines, via the celestial configuration $9\#(1,3;4,3;2,3)$, but the extra incidences 
are avoided with pseudolines using semicircular paths around the unwanted vertices. Note this drawing was also shown in \cite[Figure 3]{MaPiWi2005}; 
Theorem~\ref{thm:nonrealizeZ9} shows that this configuration is the only straight-line (weak) realization possible of the Gray graph with $\mathbb{Z}_{9}$ 
symmetry. The right-hand figure shows a pseudoline realization in which two orbits of lines are straight, and the third orbit uses pseudolines consisting of 
circular arcs in the area shown.

\begin{figure}[htbp]
\begin{center}
\begin{subfigure}[t]{.48\linewidth}
\includegraphics[width = \linewidth]{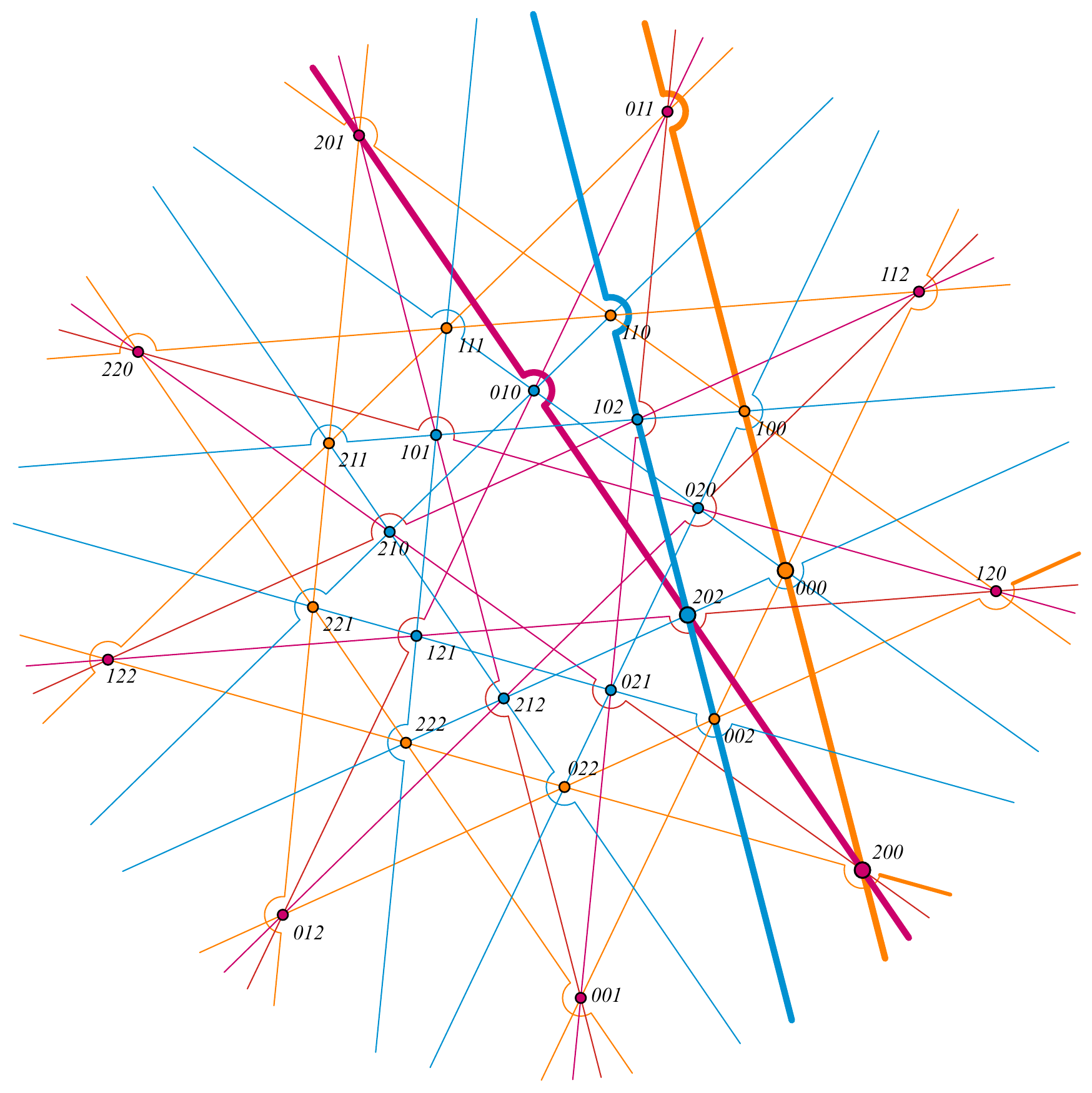}
\caption{A pseudoline drawing of the Gray configuration based on the $(27_{4})$ celestial configuration $9\#(1,3;4,3;2,3)$, avoiding the extra incidences with semicircles. 
Compare Figure 3 from \cite{MaPiWi2005}. 
}
\label{fig:Z9celestial}
\end{subfigure}
\hfill
\begin{subfigure}[t]{.48\linewidth}
\includegraphics[width = \linewidth]{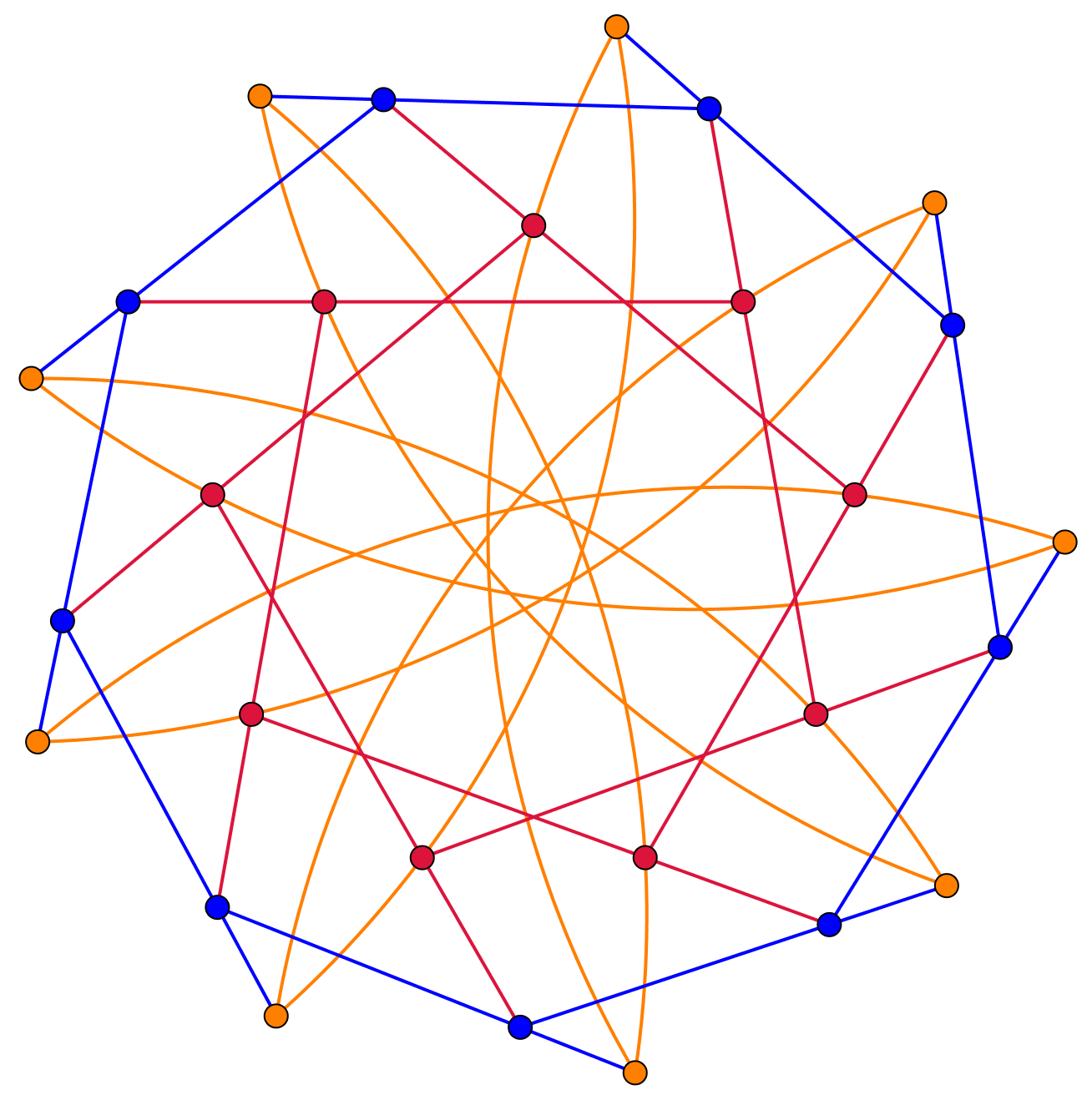}
\caption{A second pseudoline drawing, avoiding extra incidences.}
\end{subfigure}
\caption{Pseudoline realizations of the Gray configuration with $\mathbb{Z}_{9}$ symmetry.}
\label{fig:Z9realization}
\end{center}
\end{figure}

\section{Conclusion and open questions}

We have showed that among all possible polycyclic realizations of the Gray configuration, it is possible to realize both $\mathbb{Z}_{3}$ versions, 
but the $\mathbb{Z}_{9}$ realization is only topological.

There are other semisymmetric cubic bipartite graphs of girth at least 6, including the Iofinova-Ivanov graph on 110 vertices (corresponding to a $(55_{3})$ 
configuration) \cite{IvaIof1985}; the Ljubljana graph with 112 vertices (corresponding to a combinatorial $(56_{3})$ configuration \cite{CoMaMaPi2005}; the 
Tutte 12-cage, also known as the Benson Graph (corresponding to a combinatorial $(63_{3})$ configuration) with 126 vertices \cite{Ben1966, ExoJaj2008}.  
A polycyclic geometric realization of the Ljubljana configuration with 7-fold symmetry was given in \cite{CoMaMaPi2005}, and a polycyclic geometric 
realization of the configuration corresponding to the Tutte 12-cage has been shown in \cite{BoGrPi2013}. However,  although \cite{Wei} provides symmetric 
graph drawings, we are not aware of a published geometric realization of the Iofinova-Ivanov $(55_{3})$ configuration or its dual.  We plan to investigate the 
geometric realizability of the corresponding configuration; the automorphism group of the graph is isomorphic to  $\text{PGL}_2(11)$. Future work will explore 
the possibility of realizations of this configuration with 5-fold and 11-fold rotational symmetry.

Recently, Conder and Poto\v{c}nik presented a method that enabled them to generate all semisymmetric cubic bipartite graphs of order up to 10000  \cite{CoPotoapear}. 
It would be interesting if one could use our method with their census of examples to develop infinite families of geometrically realizable configurations.

 
\section*{Acknowledgements}
G\'abor G\'evay is supported by the Hungarian National Research, Development and Innovation Office, OTKA grant 
No.\ SNN 132625.
Toma\v{z} Pisanski is supported in part by the Slovenian Research Agency (research program P1-0294 and research 
projects J1-4351, J5-4596, BI-HR/23-24-012).

\bibliography{GrayPolycirculant-Feb28tmp}

\begin{thebibliography}{10}

\bibitem{Ben1966}
Clark~T. Benson.
\newblock Minimal regular graphs of girths eight and twelve.
\newblock {\em Canadian J. Math.}, 18:1091--1094, 1966.

\bibitem{Ber2008}
Leah~Wrenn Berman.
\newblock Symmetric simplicial pseudoline arrangements.
\newblock {\em Electron. J. Combin.}, 15(1):Research Paper 13, 31, 2008.

\bibitem{Ber2013}
Leah~Wrenn Berman.
\newblock Geometric constructions for 3-configurations with non-trivial
  geometric symmetry.
\newblock {\em Electron. J. Combin.}, 20(3):Paper 9, 29, 2013.

\bibitem{BerGevPis2025}
Leah~Wrenn Berman, G\'{a}bor G\'{e}vay, and Toma\v{z} Pisanski.
\newblock The {G}ray graph is a unit-distance graph.
\newblock {\em The {A}rt of {D}iscrete and {A}pplied {M}athematics}, accepted,
  2025.

\bibitem{BoGrPi2013}
Marko Boben, Branko Gr\"{u}nbaum, and Toma{\v z} Pisanski.
\newblock Multilaterals in configurations.
\newblock {\em Beitr. Algebra Geom.}, 54(1):263--275, 2013.

\bibitem{Bou1972}
Izak~Zurk Bouwer.
\newblock On edge but not vertex transitive regular graphs.
\newblock {\em J. Combin Theory Ser. B.}, 12:32--40, 1972.

\bibitem{CoMaMaPi2005}
Marston Conder, Aleksander Malni{\v c}, Dragan Maru{\v s}i\v{c}, Toma{\v z}
  Pisanski, and Primo{\v z} Poto{\v c}nik.
\newblock The edge-transitive but not vertex-transitive cubic graph on 112
  vertices.
\newblock {\em J. Graph Theory}, 50(1):25--42, 2005.

\bibitem{CoPotoapear}
Marston Conder and Primo{\v z} Poto{\v c}nik.
\newblock Edge-transitive cubic graphs: Cataloguing and enumeration.
\newblock \url{https://arxiv.org/abs/2502.02250}, 2025.

\bibitem{CoxGre}
H.~S.~M. Coxeter and S.~L. Greitzer.
\newblock {\em Geometry Revisited}.
\newblock The Mathematical Association of America, Washington, DC, 1967.

\bibitem{DoMaMa}
Ted Dobson, Aleksander Malni{\v{c}}, and Dragan Maru{\v{s}}i{\v{c}}.
\newblock {\em Symmetry in graphs}, volume 198.
\newblock Cambridge University Press, 2022.

\bibitem{ExoJaj2008}
Geoffrey Exoo and Robert Jajcay.
\newblock Dynamic cage survey.
\newblock {\em Electron. J. Combin.}, DS16(Dynamic Surveys):48, 2008.

\bibitem{Fol1967}
Jon Folkman.
\newblock Regular line-symmetric graphs.
\newblock {\em J. Combin Theory}, 3:215--232, 1962.

\bibitem{Gev2013}
G\'abor G\'evay.
\newblock Constructions for large spatial point-line $(n_k)$ configurations.
\newblock {\em Ars Math.\ Contemp.}, 7:175--199, 2013.

\bibitem{Gev2019}
G{\'a}bor G{\'e}vay.
\newblock Resolvable configurations.
\newblock {\em Discrete Appl. Math.}, 266:319--330, 2019.

\bibitem{HandDiscCompGeom2018}
Jacob~E. Goodman, Joseph O'Rourke, and Csaba~D. T\'{o}th, editors.
\newblock {\em Handbook of discrete and computational geometry}.
\newblock Discrete Mathematics and its Applications (Boca Raton). CRC Press,
  Boca Raton, FL, 2018.
\newblock Third edition of [ MR1730156].

\bibitem{GrTu}
Jonathan~L Gross and Thomas~W Tucker.
\newblock {\em Topological graph theory}.
\newblock Courier Corporation, 2001.

\bibitem{Gru2000}
Branko Gr{\"{u}}nbaum.
\newblock Connected {$(n_4)$} configurations exist for almost all {$n$}.
\newblock {\em Geombinatorics}, 10(1):24--29, 2000.

\bibitem{Gru2009b}
Branko Gr{\"u}nbaum.
\newblock {\em Configurations of {P}oints and {L}ines}, volume 103 of {\em
  Graduate Studies in Mathematics}.
\newblock American Mathematical Society, Providence, RI, 2009.

\bibitem{IvaIof1985}
A.~A. Ivanov and M.~E. Iofinova.
\newblock Biprimitive cubic graphs.
\newblock In {\em Investigations in the algebraic theory of combinatorial
  objects ({R}ussian)}, pages 123--134. Vsesoyuz. Nauchno-Issled. Inst. Sistem.
  Issled., Moscow, 1985.

\bibitem{MaPi2000}
Dragan Maru\v{s}i\v{c} and Toma{\v z} Pisanski.
\newblock The {G}ray graph revisited.
\newblock {\em J. Graph Theory}, 35:1--7, 2000.

\bibitem{MaPiWi2005}
Dragan Maru\v{s}i\v{c}, Toma{\v z} Pisanski, and Steve Wilson.
\newblock The genus of the {GRAY} graph is 7.
\newblock {\em Eur. J. Combin.}, 26:377--385, 2005.

\bibitem{Mol1990}
Emil Moln\'ar.
\newblock {\em Elemi matematika II. Geometriai transzform\'aci\'ok (Elementary
  Mathematics II. Geometric Transformations) \rm{(in Hungarian)}}.
\newblock Tank\"onyvkiad\'o, Budapest, 1990.

\bibitem{MoPi2007}
Barry Monson, Toma{\v z} Pisanski, Egon Schulte, and Asia~Ivi\'c Weiss.
\newblock Semisymmetric graphs from polytopes.
\newblock {\em J. Combin Theory Ser. A.}, 114:421--435, 2007.

\bibitem{Pis2007}
Toma{\v z} Pisanski.
\newblock Yet another look at the {G}ray graph.
\newblock {\em New Zealand J. Math.}, 36:85--92, 2007.

\bibitem{PiRa2000}
Toma{\v z} Pisanski and Milan Randi\'c.
\newblock Bridges between geometry and graph theory.
\newblock In C.~A. Gorini, editor, {\em Geometry at Work}, pages 174--194.
  Math. Assoc. America, Washington, DC, 2000.

\bibitem{PisSer2013}
Toma{\v z} Pisanski and Brigitte Servatius.
\newblock {\em Configurations from a {G}raphical {V}iewpoint}.
\newblock Birk\-h{\"a}user Advanced Texts. Birkh{\"a}user, New York, 2013.

\bibitem{Sab1958}
Gert Sabidussi.
\newblock On a class of fixed-point-free graphs.
\newblock {\em Proc. Amer. Math. Soc.}, 9:800--804, 1958.

\bibitem{sagemath}
{The Sage Developers}.
\newblock {\em {S}ageMath, the {S}age {M}athematics {S}oftware {S}ystem
  ({V}ersion 9.6)}, 2022.
\newblock {\tt https://www.sagemath.org}.

\bibitem{Wei}
Eric~W. Weisstein.
\newblock ``{I}ofinova-{I}vanov graphs.''.
\newblock \url{https://mathworld.wolfram.com/Iofinova-IvanovGraphs.html}.

\bibitem{Yag1968}
Isaak~Moiseevich Yaglom.
\newblock {\em Geometric Transformations II}.
\newblock Random House, New York, 1968.

\end{thebibliography}
\bibliographystyle{plain}

\end{document}